\documentclass{amsart}
\usepackage{amssymb}
\usepackage{stmaryrd} 
\usepackage{amsmath} 
\usepackage{amscd}
\usepackage{amsthm}
\usepackage{amsbsy}
\usepackage[margin=1in]{geometry}
\usepackage{commath, longtable}
\usepackage{comment, enumerate}
\usepackage{geometry}
\usepackage[matrix,arrow]{xy}
\usepackage{hyperref}
\usepackage{mathrsfs}
\usepackage{color}
\usepackage{float}
\usepackage{mathtools,caption}
\usepackage[table,dvipsnames]{xcolor}
\usepackage{tikz-cd}
\usetikzlibrary{shapes.geometric,calc,angles}
\usetikzlibrary{decorations.markings, arrows.meta}
\tikzset{
mid arrow/.style={postaction={decorate, decoration={
markings,
mark=at position .5 with {\arrow{Straight Barb}}
}}},
}
\tikzset{
dot/.style = {circle, fill, minimum size=#1,
              inner sep=0pt, outer sep=0pt},
dot/.default = 2pt 
}
\usepackage{asymptote}

\usepackage{longtable}
\usepackage[utf8]{inputenc}
\usepackage[OT2,T1]{fontenc}
\usepackage[normalem]{ulem}
\usepackage{hyperref}
\usepackage[customcolors]{hf-tikz}
\usepackage{enumitem}
\newlist{primenumerate}{enumerate}{1}
\setlist[primenumerate,1]{label={\arabic*$'$}}
\hypersetup{
 colorlinks=true,
 linkcolor=DarkOrchid,
 filecolor=blue,
 citecolor=olive,
 urlcolor=orange,
 pdftitle={KM graph project},
 }
\usepackage{booktabs}

\DeclareSymbolFont{cyrletters}{OT2}{wncyr}{m}{n}
\DeclareMathSymbol{\Sha}{\mathalpha}{cyrletters}{"58}

\newtheorem{theorem}{Theorem}[section]
\newtheorem{lemma}[theorem]{Lemma}

\newtheorem{corollary}[theorem]{Corollary}
\newtheorem{definition}[theorem]{Definition}

\newtheorem*{theorem*}{Theorem}
\numberwithin{equation}{section}

\theoremstyle{remark}
\newtheorem{remark}[theorem]{Remark}
\newtheorem{example}[theorem]{Example}
\newtheorem{nexample}[theorem]{Non-example}

\newcommand{\Image}{\operatorname{Im}}
\newcommand{\Div}{\operatorname{Div}}
\newcommand{\Aut}{\operatorname{Aut}}
\newcommand{\Jac}{\operatorname{Jac}}
\newcommand{\Pic}{\operatorname{Pic}}
\newcommand{\Gal}{\operatorname{Gal}}

\newcommand{\inc}{\operatorname{inc}}
\newcommand{\inv}{\operatorname{inv}}
\newcommand{\ord}{\operatorname{ord}}
\newcommand{\Char}{\operatorname{char}}
\newcommand{\unr}{\operatorname{unr}}
\newcommand{\ram}{\operatorname{ram}}

\newcommand{\rk}{\operatorname{rank}}

\newcommand{\Z}{\mathbb{Z}}
\newcommand{\E}{\mathbb{E}}
\newcommand{\I}{\mathcal{I}}
\newcommand{\cG}{\mathcal{G}}

\newcommand{\bomega}{\boldsymbol{\omega}}

\makeatletter
\theoremstyle{plain} 
\newtheorem*{intr@thm}{\intr@thmname}

\newtheorem*{c@njecture}{\conjn@name}
\newcommand{\myl@bel}[2]{
 \protected@write \@auxout {}{\string \newlabel {#1}{{#2}{\thepage}{#2}{#1}{}} }
 \hypertarget{#1}{}
 } 

\makeatother

\makeatletter
\newcommand{\mylabel}[2]{#2\def\@currentlabel{#2}\label{#1}}
\makeatother

\title[]{Iwasawa Theory of graphs and their duals}

\author[D.~Kundu]{Debanjana Kundu}
\address[DK]{University of Texas - Rio Grande Valley, USA \& University of Regina, Canada}
\email{debanjana.kundu@uregina.ca}

\author[K.~M\"uller]{Katharina M\"uller}
\address[KM]{Institut für Theoretische Informatik, Mathematik und Operations Research, Universität der Bundeswehr München, Werner-Heisenberg-Weg 39, 85577 Neubiberg, Germany}
\email{katharina.mueller@unibw.de}


\begin{document}

\begin{abstract}
In this article, we study questions pertaining to ramified $\Z_p^d$-extensions of a finite connected graph $X$.
{
Let $X_\infty$ denote a $\Z_p^d$ ramified derived graph of $X$ and $\Lambda$ denote the associated Iwasawa algebra.
We first prove an asymptotic growth formula in this general setting and show that it is possible to interpret the characteristic ideal of $\Pic_{\Lambda}(X_\infty)$ as an abstract determinant of an operator on $\Div_{\Lambda}(X_\infty^{\unr})$.}
We also {introduce the} study {of} the Iwasawa theory of dual graphs.
\end{abstract}

\maketitle

\section{Introduction}
Let $F=F_0$ be a number field and fix a prime $p$.
Consider a tower of fields
\[
F=F_0 \subset F_1\subset F_2\subset F_3\subset \dots \subset \dots F_\infty
\]
such that $F_n/F$ is Galois with $\Gal(F_n/F)\simeq \Z/p^n\Z$ and $\Gal(F_\infty/F) \simeq \varprojlim \Z/p^n\Z = \Z_p$.
K.~Iwasawa proved an asymptotic formula for the growth of the $p$-part of the class group in such a tower.
More precisely, he showed that for $n\gg 0$,
\[
{\ord_p}(\vert \textup{Cl}(K_n)\vert)=\mu p^n+\lambda n+\nu,
\]
where $\mu$, $\lambda$ are non-negative integers and $\nu$ is an integer, all of which are independent of $n$.
{Here, we write $\ord_p(\cdot)$ to mean the $p$-adic order (or the $p$-adic valuation).}

In the last couple of years this theory was generalized to unramified coverings of finite connected graphs; see for example \cite{gonet2022iwasawa, Val21, dubose-vallieres, kataoka2023fitting, kleine2023non, MV23, MV24, DLRV_Nagoya, Kleine-Mueller4}.
More recently R.~Gambherra--D.~Vallierés \cite{GV24} considered \textit{branched} $\Z_p$-coverings of finite connected graphs, i.e. coverings where at least one of the vertices is ramified.
Let
\[
X \longleftarrow X_1 \longleftarrow X_2 \longleftarrow \dots \longleftarrow X_n \longleftarrow \cdots
\]
be a tower of branched coverings such that the group $\Z/p^n\Z$ acts without inversion on $X_n/X$. 
{Let $\pi_n\colon X_n\to X$ be the covering map.
We always assume that $\pi_n\circ g=\pi_n$ for all $g\in \Z/p^n\Z$ and to simplify notation we write $\Z/p^n\Z$ acts on $X_n/X$. }
Assume that all graphs $X_n$ are connected and denote the number of spanning trees of $X_n$ by $\kappa(X_n)$, then
\[
\ord_p(\kappa(X_n))=\mu p^n +\lambda n+\nu \text{ for } n\gg 0.
\]

For a finite connected graph, the Kirchhoff Matrix Theorem asserts that the number of spanning trees coincides with the size of the Jacobian; see Section~\ref{Jacobian sec} for a precise definition.
If a group $G$ acts on $X_n/X$ it also acts naturally on the Jacobian, denoted by $\Jac(X_n)$, and the Picard group, denoted by $\Pic(X_n)$, turning both these groups into Iwasawa modules.
Assume that $X_\infty/X$ is a branched covering obtained from a voltage assignment; see Section~\ref{sec: IT of pro-p tower} for details on how such a graph is constructed.
Each such voltage assignment also induces an unramified tower $X_\infty^{\unr}/X$ together with an immersion
\[
\iota \colon X_\infty^{\unr}\to X_\infty.
\]
In \cite{GV24} the authors describe the characteristic ideal of $\varprojlim_n \Pic(X_n)\otimes \Z_p$ in terms of a linear operator defined on $\varprojlim_n \Div(X_n)\otimes \Z_p$.
In the special case when $X_\infty = X_\infty^{\unr}$ is an unramified cover of $X$ their result recovers the Iwasawa Main Conjecture (IMC) proved by the second named author and S.~Kleine in  \cite{Kleine-Mueller4}.

\subsection*{Our results}
The first objective of the present article is to generalize the aforementioned theorem of branched coverings to $\Z_p^d$-coverings\footnote{{By this we mean $d$-many copies of $\Z_p$. If there arises a need for clarification we may also refer to this as $\Z_p^{\oplus d}$.}} of $X$ where $d\ge 1$; see Theorem~\ref{thM.char-ideals}.

\begin{theorem*}
Let $X_\infty=X(\Z_p^d,\mathcal{I},\alpha)$ denote a (ramified) derived graph of $X$.
Then there exists an explicit linear operator $\Delta$ on $\Pic_\Lambda(X_\infty^{\unr})$ such that 
\[
\Char_{\Lambda}(\Pic_\Lambda(X_\infty))=\det(\Delta).\]
\end{theorem*}

Furthermore, we prove an asymptotic formula for the number of spanning trees of the intermediate graphs of $\Z_p^d$-towers in Theorem~\ref{asymp formula} when $d\geq 2$.

\begin{theorem*}
Let $X_\infty$ be as in the previous theorem and let $X_n/X$ be the intermediate graphs.
Then there exist non-negative constants $\mu$ and $\lambda$ (independent of $n$) such that 
\[
\ord_p(\kappa(X_n) )=\mu p^{nd}+\lambda np^{n(d-1)}+O(p^{n(d-1)}).
\]
\end{theorem*}

In the last section of this paper, we initiate the study of towers of dual (planar) graphs; see Theorem~\ref{dual graph theorem}.

\begin{theorem*}
Let $X_\infty/X$ be a branched $\mathcal{G}$-tower, where $\mathcal{G}$ is a uniform pro-$p$ $p$-adic Lie group.
Assume that all $X_n$ are planar.
If $X_n^\vee/X^\vee$ is a branched covering for all $n$, then $X_\infty^\vee/X^\vee$ is also a $\mathcal{G}$-tower.
\end{theorem*}

{We observe that upon taking the duals of (planar) coverings of a graph, we may not obtain a branched covering of dual graphs.}
In Section~\ref{sec app} we provide conditions to ensure that the hypothesis of the above theorem are satisfied.
{In particular, we show that $X_\infty$ can have \emph{at most} two ramified vertices if we want the dual tower to also be a branched cover of $X^\vee$.}
    
\subsection*{Future Direction}
This topic of Iwasawa theory of graphs, and in particular branched coverings of finite connected graphs leaves plenty of room for further investigations.
\begin{itemize}
\item Even though the operator $\Delta$ describing the characteristic ideal of $\Pic_\Lambda(X_\infty)$ generalizes the linear operator appearing in the IMC for the unramified case, it is not clear whether this operator can be interpreted as a $p$-adic $L$-function in the branched setting as well.
In other words, so far the analytic side of an IMC does not exist in the setting of branched covers.

\item The asymptotic formula for the number of spanning trees in a branched $\Z_p^d$-covering involves an error term of size $O(p^{n(d-1)})$.
This error term does not occur in the unramified setting: for an unramified $\Z_p^d$-covering there exists a polynomial $P(X,Y)\in \Z[X,Y]$ of total degree $d$ and degree at most $1$ in $Y$ such that 
\[
\ord_p(\kappa(X_n))=P(p^n,n).
\]
In a future project the authors intend to investigate whether it is possible to improve the error term.

\item It would also be interesting to study a Kida-type formula for the branched $\Z_p^d$-covers. 
 
\end{itemize}

\subsection*{Organization}
Section~\ref{sec: prelim} is preliminary in nature.
In this section we collect all the definitions and notations which will be used throughout the article.
In Section~\ref{sec: planar derived graphs} we prove our first result which provides sufficient conditions for the derived graph of a planar graph to be planar.
Section~\ref{sec: towers of derived graphs} concerns proving Iwasawa theory results for $\Z_p^d$ towers of branched covers of $X$.
We prove an asymptotic growth formula in the general setting.
We also show that it is possible to interpret the characteristic ideal of $\Pic_{\Lambda}(X_\infty)$ as an abstract determinant of an operator on $\Div_{\Lambda}(X_\infty^{\unr})$.
In Section~\ref{sec: results on dual graphs} we introduce the study of towers of dual graphs.
Throughout this article, we elucidate our results with explicit examples.

\subsection*{Acknowledgements}
The authors thank Rusiru Gambheera, Antonio Lei, and Daniel Vallier{\'e}s for comments on an earlier draft of the paper and for helpful discussions.
{We thank Karen Maeghar for answering our questions.
DK acknowledges the support of an early career AMS--Simons travel grant.
We thank the referee for their comments and suggestions which helped improve the exposition and simplify the proof of Theorem~\ref{thm:planar}.}

\section{Preliminaries}
\label{sec: prelim}

In this section we collect definitions and set notations that are required for the remainder of the article.
We also record some elementary lemmas in this section.

\subsection{Basic Definitions in Graph Theory}

A graph $X$ consists of a vertex set $V(X)$ and a set of \emph{directed edges} $\E(X)$ along with an \emph{incidence function} 
\begin{align*}
\inc \colon \E(X) &\longrightarrow V(X) \times V(X)\\
e &\mapsto (o(e), t(e))
\end{align*}
and an \emph{inversion function} 
\begin{align*}
\inv \colon \E(X) &\longrightarrow \E(X)\\
e &\mapsto \overline{e} 
\end{align*}
satisfying the following conditions for all $e \in \E(X)$ 
\begin{enumerate}
\item[(i)] $\overline{e} \neq e$ 
\item[(ii)] $\overline{\overline{e}} = e$
\item[(iii)] $o(\overline{e}) = t(e)$ and $t(\overline{e}) = o(e)$. 
\end{enumerate}
The vertex $o(e)$ is called the \emph{origin} and the vertex $t(e)$ is called the \emph{terminus} of the directed edge $e$.
Given $v \in V(X)$, write 
\begin{align*}
\E^o_v(X) & = \{e \in \E(X) : o(e) =v\}\\
\E^t_v(X) & = \{e \in \E(X) : t(e) =v\}.
\end{align*}
The inversion map induces the bijection 
\[
\inv \colon \E^o_v(X) \longrightarrow \E^t_v(X) \textrm{ for all } v \in V(X).
\]
The set of \emph{undirected edges} is obtained by identifying $e$ with $\overline{e}$ and is denoted by $E(X)$.
Finally, write $E_v(X)$ to denote the set of all undirected edges adjacent to $v$. 
When it is clear from the context which graph is being referred to, we often drop it from the notation of $E$, $E_v$, $\E^o_v$, $\E^t_v$, and $V$.

\begin{definition}
Let $X$ and $Y$ be directed graphs.
A \emph{morphism} $f \colon Y \to X$ of graphs consists of two functions $f_V \colon V(Y) \to V(X)$ and $f_{\E} \colon \E(Y) \to \E(X)$ which satisfy the following properties for all $e \in \E(Y)$
\begin{enumerate}
    \item[\textup{(}i\textup{)}] $f_V(o(e)) = o(f_{\E}(e))$, 
    \item[\textup{(}ii\textup{)}] $f_V(t(e)) = t(f_{\E}(e))$, 
    \item[\textup{(}iii\textup{)}] $\overline{f_{\E}(e)} = f_{\E}(\overline{e})$.
\end{enumerate}
\end{definition}

We write $f$ to denote both $f_V$ and $f_{\E}$ when there is no chance of confusion.
Given a morphism of graphs $f \colon Y \to X$ and any vertex $v \in V(Y)$, the restriction $f\vert_{\E^{o}_v(Y)}$ induces a function \[
f\vert_{\E^o_v(Y)} \colon \E^o_v(Y) \longrightarrow \E^o_{f(v)}(X).
\]

\subsection{Jacobian of Graphs}
\label{Jacobian sec}
\begin{definition}
A \emph{divisor} on a (possibly infinite\footnote{When we consider infinite graphs we require that the graph is locally finite.}) graph $X$ is an element of the free abelian group on the vertices $V (X)$ defined as
\[
\Div(X) = \left\{ \sum_{v\in V (X)}
a_v v \mid  a_v \in \Z\right\}
\]
where each $\sum_{v\in V (X)}a_v v$ is a formal linear combination of the vertices of $X$ with integer coefficients and only finitely many $a_v$ are non-zero (when $X$ is an infinite graph).


The \emph{degree homomorphism} is defined as
\begin{align*}
\deg \colon \Div(X) &\longrightarrow \Z\\
\sum_{v\in V (X)}a_v v &\mapsto \sum_{v\in V(X)} a_v.
\end{align*}
The kernel of this degree map is the subgroup of divisors of degree 0, and is denoted as $\Div^0(X)$.
\end{definition}

{Write $\mathcal{M}(X)$ to denote the set of $\Z$-valued functions on $V(X)$.
Define the function $\psi_v\in \mathcal{M}(X)$ as follows:
\[
\psi_v(v_0) = \begin{cases}
1 & \text{ if } v = v_0 \\
0 & \text{ if } v \neq v_0.
\end{cases}
\]
Note that $\psi_v$ forms a $\Z$-basis of $\mathcal{M}(X)$ as $v$ runs over all vertices $v\in V(X)$.
We can then define a group morphism
\begin{align*}
    \operatorname{div}: \mathcal{M}(X) & \longrightarrow \Div(X)\\
    \psi_{v_0} & \mapsto \sum_{v\in V(X)} \rho_{v}(v_0) \cdot v
\end{align*}
where
\[
\rho_v(v_0) = \begin{cases}
    \abs{\E_{v_0}^o(X)} - 2\times \#(\text{undirected loops at }v_0) & \text{ if } v=v_0\\
    -\#(\text{undirected edges from }v \text{ to }v_0) & \text{ if } v\neq v_0.
\end{cases}
\]}

\begin{definition}
{With notation as introduced above, define the \emph{principal divisors on $X$} as
\[
\Pr(X) = \operatorname{div}(\mathcal{M}(X)).
\]}
The \emph{Picard group of $X$} is then defined as
\[
\Pic(X) = \Div(X)/\Pr(X).
\]
The \emph{Jacobian group of $X$} is defined as
\[
\Jac(X) = \Div^0(X)/\Pr(X).
\]
{In the literature, this also referred to as \emph{Picard group of degree zero} and denoted by $\Pic^0(X)$.}
\end{definition}

\subsection{Group action on graphs, Galois theory, and coverings of graphs}

Write $\Aut(X)$ to denote the group of automorphisms of a graph $X$. 
\begin{definition}
Let $G$ be a group and $X$ be a graph.
The group $G$ is said to \emph{act on $X$} if  {$G$ acts on $V(X)$ and $\E(X)$ as follows:
for $g\in G$, if $e$ is an edge from $v$ to $w$ , then $g\cdot e$ is an edge from $g\cdot v$ to $g\cdot w$.}
The group $G$ is said to \emph{act without inversion} if for all $e \in \E(X)$ and all $g \in G$, one has $g \cdot e\neq \overline{e}$.
\end{definition}

When a group $G$ acts on a graph $X$, it is understood that $G$ acts on both the set of vertices and edges.
If $G$ acts on a graph $X$ such that the action on $V(X)$ is free, then it acts freely on the set of edges as well.
{The converse however is not true.}

\begin{lemma}
Let $X$ be a locally finite graph, and let $G$ be a group acting on $X$ without inversion.
Then, $\Jac(X)$ is a $\Z[G]$-module.
\end{lemma}

\begin{proof}
See \cite[Corollary~2.6]{GV24}.    
\end{proof}

\begin{definition}
Let $f\colon Y\to X$ be morphism of directed graphs, that is surjective on the set of edges and on the set of vertices. 
The morphism $f$ is a \emph{branched cover} if for each vertex $v\in V(Y)$, the cardinality $m_v= (f\vert_{\E^o_v}^{-1}(e))$ does not depend on the choice of $e\in \E^o_{f(v)}(X)$.
The cardinality $m_v$ is called the \emph{ramification index} at $v$.
If $m_v=1$, the morphism $f$ is said to be an \emph{unramified cover}. 

\smallskip

The function $f\vert_{\E^o_v} \colon \E^o_v(Y) \to \E^o_{f(v)}(X)$ is $m_v$-to-$1$.
The graph $Y$ is called a \emph{$d$-sheeted covering} of a connected graph $X$ if for all $w\in V(X)$
\[
d = [Y:X] = \sum_{v \in f^{-1}(w)} m_{v}
\]
\end{definition}

When $Y$ is a $d$-sheeted unramified cover of a connected graph $X$, then for all $w\in V(X)$
\[
d = [Y:X] =  \sum_{v\in f^{-1}(w)} 1 = \# \{v \in V(Y) \mid f(v) = w\}.
\]

\begin{definition}
{Let $Y$ be a covering of $X$ with projection map $\pi \colon Y \to X$.}
\begin{enumerate}
\item[\textup{(}i\textup{)}] \emph{Galois group of unramified covering:} If $Y/X$ is a $d$-sheeted \emph{unramified} covering 
such that there are exactly $d$ graph-automorphisms $\sigma \colon Y \rightarrow Y$ satisfying $\pi \circ \sigma = \pi$, {such a covering} is {called} \emph{normal} or \emph{Galois}.
In this case, the Galois group is $G$ and is given by
\[
G = \Gal(Y/X) = \{\sigma \colon Y \rightarrow Y \mid \pi \circ \sigma = \pi\}.
\]

\item[\textup{(}ii\textup{)}] \emph{Galois group of branched covering:}
If $Y/X$ is a \emph{branched} covering and $G$ is a group that acts freely without inversion on $\E(Y)$, acts trivially on $\E(X)$, and is compatible with the covering $\pi\colon Y\to X$, {such a branched covering} $Y/X$ is called \emph{Galois} {and the corresponding} Galois group is $G$.
\end{enumerate}
\end{definition}

\begin{remark}
{Another way to think about unramified Galois covers is by considering the Galois group as a deck transformation.
An unramified Galois cover of a finite (connected) graph $X$ can be constructed by making a finite group $G$ act freely on $V(X)$ and without inversion.
Note that under these hypotheses, the notion of a quotient graph $G\backslash X$ makes sense.
The natural morphism of graphs $X \to G\backslash X$ is an unramified Galois cover with group of deck transformations isomorphic to $G$.
The notion of branched covering relaxes the condition that the action of $G$ on vertices is free but still requires that $G$ acts freely (and without inversion) on directed edges.}
\end{remark}

Let $Y/X$ be an abelian covering of connected graphs; i.e., we assume that its Galois group $G$ is abelian.
Then the Jacobian group $\Jac(Y)$ is naturally equipped with a $\Z[G]$-module structure.

\subsection{Iwasawa theory of pro-\texorpdfstring{$p$}{} tower of graphs obtained from a voltage assignment}
\label{sec: IT of pro-p tower}
We describe the construction of uniform pro-$p$ $p$-adic Lie extensions of connected graphs.
Let $\cG$ be a (possibly non-Abelian) uniform pro-$p$ $p$-adic Lie group of dimension $d$.
This means $\cG$ is topologically generated by $d$ elements and has a filtration 
\[
\cG \supseteq \cG^p \supseteq \cG^{p^2} \supseteq \cdots \supseteq \cG^{p^n} \supseteq \cdots 
\]
such that $\cG^{p^{n+1}} \trianglelefteq \cG^{p^n}$ and each subsequent quotient is isomorphic to $(\Z/p\Z)^d$.
{When $\cG$ is abelian, we write $\cG = \Gamma^{\oplus d} \simeq \Z^{\oplus d}_p$ where $\Gamma$ is a multiplicative group topologically isomorphic to $\Z_p$.}

Let $X$ be a finite connected graph with vertex set $V(X)$ and set of edges $\E(X)$.
We define a \textit{voltage assignment} function as follows 
\[
\alpha \colon \E(X) \longrightarrow \cG \textrm{ satisfying  } \alpha(\overline{e}) = \alpha(e)^{-1}.
\]
For each vertex $v\in V(X)$, choose a closed subgroup $I_v$ of $\cG$, and define the set 
\[
\I = \{(v,I_v) : v \in V(X)\}.
\]
This gives a graph $X_\infty = X(\cG,\I,\alpha)$ with vertex set equal to the disjoint union
\[
V(X_\infty) = \bigsqcup_{v\in V(X)} \{v\} \times \cG/I_v
\] and the collection of directed edges is given by 
\[
\E(X_\infty)=\E(X)\times \cG.
\]
Let $g\in \cG$.
The directed edge $(e,g)$ connects the vertex $(o(e),gI_{o(e)})$ to the vertex $(t(e), g \alpha(e)I_{t(e)})$.
Finally, let $\overline{(e, g)} = (\overline{e},g \alpha(e))$.
This is a branched cover $X_\infty \to X$.
Moreover, the graph $X_\infty$ is infinite.

In order to get finite graphs, for each positive integer $n$, consider the natural surjective group morphism $\pi_n \colon \cG \twoheadrightarrow \cG_n$, where $\cG_n = \cG/\cG^{p^n}$.
Set $\alpha_n = \pi_n \circ \alpha \colon \E(X) \to \cG_n$ and define subgroups 
\[
\I_n = \{(v,\pi_n(I_v)) \ : \ v \in V(X)\}
\]
of $\cG_n$ indexed by the vertices of $X$.
This gives a family of graphs $X_n =X(\cG_n,I_n,\alpha_n)$ which are finite.
The natural surjective group morphisms $\cG_{n+1} \twoheadrightarrow \cG_n$ induce branched covers $X_{n+1} \to X_n$ for each non-negative integer $n$.
Therefore, we obtain a (Galois) tower of graphs 
\[
X = X_0 \longleftarrow X_1 \longleftarrow X_2 \longleftarrow \cdots \longleftarrow X_n \longleftarrow \cdots \longleftarrow X_\infty, 
\]
where each map $X_{n+1} \to X_n$ is a branched cover satisfying $[X_{n+1} : X_n] = p^d$.
Such a tower will be referred to as a \emph{branched $\cG$-tower of finite graphs} provided that all $X_n$ are connected.

When the closed subgroups $I_v$ are all trivial, the graph $X_\infty = X(\cG,\alpha)$ is the \emph{unramified $\cG$-tower}.

\begin{lemma}
\label{lemma:number of vertices}
Let $(X_n)_{n\in \mathbb{N}}$ be a $\mathcal{G}$-tower of graphs constructed as above and assume that all $X_n$ are connected.
Suppose that there exists an index $n_0$ such that the number of ramified vertices in $X_\infty/X_{n_0}$ stabilizes.
Let {$\mathsf{v}_r$} be the number of ramified vertices and let {$\mathsf{v}_u$} be the number of unramified vertices of $X_\infty/X_{n_0}$.
Then the number of vertices of $X_n$ for $n\ge n_0 $ is given by
\[
\mathsf{v}_r+\frac{\vert \mathcal{G}_n\vert }{\vert \mathcal{G}_{n_0}\vert } \mathsf{v}_u.
\]
\end{lemma}

\begin{proof}
When $n\ge n_0$, the vertices of $X_n$ ramified in $X_\infty/X_n$ and the ones ramified in $X_\infty/X_{n_0}$ are in one-to-one correspondence by the definition of $n_0$.
The pre-images of the unramified vertices are given by $\{(v,g)\mid g\in \mathcal{G}^{p^{n_0}}/\mathcal{G}^{p^n}\}$. 
The claim follows. 
\end{proof}

\subsection{Matrices associated to graphs}
We now remind the reader of some matrices associated to graphs that will be important throughout the discussion.
{We often prefer to work with a multiplicative notation, so we write $\Gamma$ in place of $\Z_p$.}

\begin{definition}
\label{matrix defn}
Let $X$ be a finite connected graph.
{Let $\alpha$ be a voltage assignment on $\E(X)$ with values in $\cG = {\Gamma^{\oplus d} \simeq} \Z_p^d$.
Let $R\subset V(X)$ be a subset.
For each vertex $v\in R$, fix a non-trivial subgroup $I_v\subset \mathcal{G}$.
The vertices in $R$ are called \emph{ramified} and the vertices in $V(X)\setminus R$ are called \emph{unramified}.
If $I_v\cong {\Gamma}$, fix a topological generator $\sigma_v$ of $I_v$ and define $\widetilde{T}_v=\sigma_v-1$.
Otherwise set $\widetilde{T}_v=1$. }

\begin{enumerate}
\item[\textup{(}i\textup{)}] The \emph{adjacency matrix} $A = {A_X} = \left(a_{i,j}\right)$ is defined as one where
\[
a_{i,j} = \begin{cases}
\textrm{twice the number of undirected loops at }v_i  & \textrm{ when }i=j\\
\textrm{number of undirected edges connecting the $v_i$ to $v_j$} & \textrm{ when }i\neq j.
\end{cases}
\]
\item[\textup{(}ii\textup{)}] The \emph{valency matrix} $D= {D^R_X} = (d_{i,j})$ is a diagonal matrix where 
\[
d_{i,j} = \begin{cases}
\abs{\E_{v_i}^o(X)} & \textrm{ when }i=j \textrm{ and } v_i \textrm{ is an unramified vertex}\\
0 & \textrm{ when }i\neq j \textrm{ or } v_i \textrm{ is a ramified vertex}.
\end{cases}
\]
\item[\textup{(}iii\textup{)}]  {If $R=\emptyset$, }the matrix $D - A$ is called the \emph{Laplacian matrix}.
\item[\textup{(}iv\textup{)}]
Suppose that there are $s$ many vertices in $X$ arranged such that $v_1, \ldots, v_{r}$ are unramified vertices and $v_{r+1}, \ldots, v_s$ are the ramified vertices.
The matrix $B = B^{{R,\alpha}}_{{X}}= (b_{i,j})\in M_{s\times s}({\Z_p\llbracket \cG \rrbracket})$ is defined as follows
\[
b_{i,j} = \begin{cases}
    - \widetilde{T}_{v_i} &\textrm{ when } r+1 \le i = j \le s\\
     \displaystyle\sum_{\substack{e\in \E(X) \\ \inc(e)=(v_j, v_i)}}{\alpha(e)} & \textrm{ when } 1\le i \le s \textrm{ and } 1\le j \le r\\
     0 & \textrm{otherwise}.
\end{cases}
\]
\end{enumerate} 
\end{definition}


\subsection{Growth patterns in \texorpdfstring{$\Z_p^d$}{}-towers of graphs}
A strong analogy between number theory and graph theory has been observed by mathematicians.
This suggests that it might be possible to study the variation of the $p$-part of the number of spanning trees as one goes up a $\Z^d_p$-tower of graphs.
In fact, there is a perfect analogue of Iwasawa’s asymptotic class number formula in the setting of graph theory \cite{Val21, gonet2022iwasawa}.
For example, when $d=1$ there exist non-negative integers $\mu, \lambda, n_0$, and $\nu \in \Z$ such that for $n\geq n_0$
\[
\ord_p(\kappa(X_n)) = \mu p^n + \lambda n + \nu
\]
where $\kappa(X_n)$ is the number of spanning trees of $X_n$.

In the following paragraphs we summarize some relevant notions from \cite{Kleine-Mueller4}.
Henceforth, let $\cG = {\Gamma^{\oplus d} \simeq} \Z_p^{\oplus d}$ for some $d\geq 1$.
As discussed previously, to every finite connected graph $X$ we can attach a finite abelian group called the Jacobian of $X$.
Moreover, any cover of finite connected graphs $f \colon Y \to X$ induces a surjective group morphism of their respective Jacobian, i.e., there is a surjective map
\[
f_* \colon \Jac(Y) \longrightarrow \Jac(X).
\]

Starting with a $\Z^d_p$-tower of a finite connected graph $X$, we obtain a compatible system of maps between the Sylow-$p$ subgroups of the Jacobians, i.e.,
\[
\Jac(X_{n+1})[p^\infty] \longrightarrow \Jac(X_n)[p^\infty].
\]
Recall that $\Jac(X_n)[p^\infty]$ is a $\Z_p[(\Z/p^n\Z)^d]$-module.
Since $X$ is assumed to be finite and connected,
\[
\Jac(X)[p^\infty] \simeq \Z_p \otimes_{\Z} \Jac(X);
\]
the same holds for each subsequent layer $X_n$.
We then define the inverse limit 
\[
\Jac_{\Lambda}(X_\infty) := \varprojlim_{n} \Jac(X_n)[p^\infty] = \varprojlim_{n} \left( \Z_p \otimes_{\Z} \Jac(X_n) \right);
\]
this is a finitely generated torsion module over the Iwasawa algebra $\Lambda=\Z_p\llbracket \mathcal{G}\rrbracket$.

It is well-known \cite{Gre73} that there is a non-canonical isomorphism 
\begin{align*}
{\Lambda = }\Z_p\llbracket \cG \rrbracket &\simeq \Z_p \llbracket T_1,\dots, T_d \rrbracket\\
\gamma_i & \mapsto 1+T_i,
\end{align*}
where $\gamma_1,\dots ,\gamma_d$ are topological generators of $\mathcal{G} = {\Gamma^{\oplus d}}$.
{Henceforth, we will go between the Iwasawa algebra and its identification with the formal power series ring in $d$-variables as is convenient.}

{A finitely generated, torsion $\Lambda$-module $M$ is said to be \emph{pseudo-null} if its annihilator $\operatorname{Ann}_{\Lambda}(M)$ has height at least 2. Equivalently, $\operatorname{Ann}_{\Lambda}(M)$ is generated by two co-prime elements.
The \emph{structure theorem} of finitely generated $\Lambda$-modules asserts that there is a \emph{pseudo-isomorphism} (i.e., homomorphism such that kernel and cokernel are pseudo-null) from $M$ to a (unique) module $\bigoplus_{i=1}^t \Lambda/\langle f_i^{e_i}\rangle$ where $f_i\in \Lambda$ are irreducible.}

Viewing $\Jac_{\Lambda}(X_\infty)$ as a $\Z_p\llbracket T_1,\dots, T_d \rrbracket$-module, and using the structure theorem we can define the \textit{characteristic ideal} ${\Char}_{\Lambda}(\Jac_{\Lambda}(X_\infty))$.
In other words,
\begin{align}\label{pseudooo}
{\Char}_{\Lambda}(\Jac_{\Lambda}(X_\infty)) = \langle f_{X}(T_1, \ldots, T_d) \rangle,
\end{align}
where $f_X(T_1, \ldots, T_d)$ is an element of $\Lambda$ and is a generator of the characteristic ideal {(determined up to a unit)}; it is called the \emph{characteristic element} of $\Jac_{\Lambda}(X_\infty)$.

\textbf{Notation:}
We write $\Div_{\Lambda}(X_\infty)$ to denote $\Div(X_\infty) \otimes_{{\Z_p[\mathcal{G}]}} {\Lambda}$ 
and similarly for $\Div^0_{\Lambda}(X_\infty)$, $\Pr_{\Lambda}(X_\infty)$.

\section{Planar Derived Graphs}
\label{sec: planar derived graphs}

We start by recalling the construction of an unramified \emph{derived graph}, which is denoted by $Y=X(G,\alpha)$.

\subsubsection*{Construction:}
Let $X$ be a finite and connected graph.
Suppose that the voltage group $G$ is finite of size $N$.
Label the elements in $G$ as $\tau_0, \cdots, \tau_{N-1}$.

To get the vertices of $Y$, consider $N$ copies of each vertex $x \in V (X)$ and label them as $x_{\tau_0}, \cdots, x_{\tau_{N-1}}$.
Therefore, the derived graph $Y$ has $N \abs{X}$ many vertices.
Suppose that $x \to x'$ is a directed edge in $X$ with voltage assignment $\alpha(e_{x,x'}) = \tau \in G$.
For every $\tau_i$ in $G$, there is an edge in $Y$ joining $x_{\tau_i}$ to $x'_{\tau_i \cdot \tau}$.
This process is then repeated for every $x\in V(X)$.

When $G$ is infinite (as will be in our case) we can repeat the same construction formally.

\subsubsection*{Intuition:}
A voltage graph takes $\abs{G}$-many copies of the base graph $X$, stacked vertically above each other.
Then for each directed edge $x \to x'$ with voltage $\tau$, we do the following: for each sheet indexed by $\sigma \in G$, we join the vertex $x_{\sigma}$ to the vertex $x'_{\sigma \cdot \tau}$.
So if we ``look at the layer of sheets from above'' we see vertex $x$ joined to vertex $x'$ exactly $\abs{G}$-many times.

\bigskip

The main goal of this section is to provide a sufficient condition for the derived graph of a planar graph to be planar.
It is easy to construct examples of planar graphs and voltage assignments, such that the derived graph is \textit{not} planar.
We provide an example below
\begin{example}
Let $X$ be a bouquet graph with two loops, i.e. a graph consisting of a single vertex with two loops $e_1,e_2$. Let $G=(\Z/5\Z,+)$ and define $\alpha(e_1)=1$, $\alpha(e_2)=2$. Let $Y=X(G,\alpha)$. Then $Y$ is the complete graph on $5$ vertices which is not planar.

\begin{center}
\begin{tikzpicture}[scale=1]
\node[inner sep=1pt] (A) at (-2,0) {\tiny{$A$}};

\fill (-1.85,0) circle (1.25pt);

\path[scale=2] 
        (A.east) edge [out=30, in=-30, distance=5mm, -, thick, "$e_1$" ](A.east);
\path[scale=4] 
        (A.east) edge [out=30, in=-30, distance=5mm, -, "$e_2$", thick, red] (A.east);
\end{tikzpicture}
\hspace{1cm}
    \begin{tikzpicture}[
                    ]
\node (n)   [regular polygon, 
             regular polygon sides=5, minimum size=22mm,
             draw=black, thick]  {};
\foreach \i in {1,2,...,5} 
{     
    \foreach \j [evaluate=\j as \k using int(\i+\j)] in {2,3}
    {
    \ifnum\k<6
    \draw[thick,red]   (n.corner \i) -- (n.corner \k);
    \fi
    }
\node [dot] at (n.corner \i) {};
}
    \end{tikzpicture}
\end{center}
\end{example}

\begin{theorem}
\label{thm:planar}
Fix a finite group $G$.
Let $X$ be a finite undirected planar graph and $Y$ be the undirected derived graph of $X$ of a single\footnote{By this we mean that exactly one of the edges has a non-trivial voltage assignment.} \emph{unramified} voltage assignment.
Then $Y$ is planar.
\end{theorem}

\begin{proof}
Let $e_0$ be the edge with the non-trivial voltage assignment $\alpha$.
Let $v$ and $w$ be the vertices joined by $e_0$.
It suffices to show that each connected component of $Y$ is planar.
Therefore, we can without loss of generality assume that $G$ is generated by $\alpha(e_0)$ where $\alpha$ is the voltage assignment.

{Using a stereographic projection we can assume that $e_0$ lies on the outer face of $X$.
Let $X'$ be the graph obtained from $X$ by removing the edge $e_0$.
Let $Y'$ be $\vert G\vert$ copies of $X'$ indexed by $\alpha(e_0)^i$. We call the different copies of $X'$ the sheets of $Y'$. As $X'$ is planar the same is clearly true for $Y'$. The graph $Y$ is now obtained from $Y'$ by connecting the vertex $v$ in the $\alpha(e_0)^i$ sheet with the vertex $w$ in the $\alpha(e_0)^{i+1}$-sheet. As $e_0$ is an edge on the outer face this is still a planar graph}\footnote{{We thank the anonymous referee for pointing out this simplified proof to us.}}.
\end{proof}

We explain the above theorem via two explicit examples.
The first example is rather simple but might be helpful for the reader to understand the terms defined in the proof of the above theorem.
The nuances of the proof become more clear with the second example.

\begin{example}
\label{unramified derived graph}
Let the voltage group $G=\{\mathbf{1},\tau\}$.
Consider the finite planar graph $X$ where exactly one edge (namely $AB$) has a non-trivial voltage assignment $\alpha(e_{AB}) = \tau$.
\begin{center}
\begin{tikzpicture}[scale=1]
\node[inner sep=0pt, shape = circle] (A) at (0,1.5) {\tiny{$A$}};
\node[inner sep=0pt, shape = circle] (B) at (-1.5,0) {\tiny{$B$}}; 
\node[inner sep=0pt, shape = circle] (C) at (1.5,0) {\tiny{$C$}}; 
\node[inner sep=0pt, shape = circle] (T) at (-0.82,0.75) {$\tau$};

\fill (0,1.35) circle (1.5pt);
\fill (1.35,0) circle (1.5pt);
\fill (-1.35,0) circle (1.5pt);

\draw[thick, mid arrow, dashed] (A.south) to (B.east);
\draw[thick, mid arrow] (C.west) to (A.south);
\draw[thick, mid arrow] (B.east) to (C.west);
\end{tikzpicture}
\end{center}
Here $e_0 = e_{AB}$.
{In this case $X'$ is following graph}
\begin{center}
\begin{tikzpicture}[scale=1]
\node[inner sep=0pt, shape = circle] (A) at (0,1.5) {\tiny{$A$}};
\node[inner sep=0pt, shape = circle] (B) at (-1.5,0) {\tiny{$B$}}; 
\node[inner sep=0pt, shape = circle] (C) at (1.5,0) {\tiny{$C$}};

\fill (0,1.35) circle (1.5pt);
\fill (1.35,0) circle (1.5pt);
\fill (-1.35,0) circle (1.5pt);

\draw[thick, mid arrow] (C.west) to (A.south);
\draw[thick, mid arrow] (B.east) to (C.west);
\end{tikzpicture}
\end{center}
{In the case $G=\Z/2\Z$ the graph $Y'$ has the following shape:}
\begin{center}
\begin{tikzpicture}[scale=1]
\node[inner sep=0pt, shape = circle] (A) at (0,1.5) {\tiny{$A_1$}};
\node[inner sep=0pt, shape = circle] (B) at (-1.5,0) {\tiny{$B_1$}}; 
\node[inner sep=0pt, shape = circle] (C) at (1.5,0) {\tiny{$C_1$}}; 
\node[inner sep=0pt, shape = circle] (A') at (5.5,1.5) {\tiny{$A_2$}};
\node[inner sep=0pt, shape = circle] (B') at (4,0) {\tiny{$B_2$}}; 
\node[inner sep=0pt, shape = circle] (C') at (7,0) {\tiny{$C_2$}}; 

\fill (0,1.33) circle (1.5pt);
\fill (1.3,0) circle (1.5pt);
\fill (-1.3,0) circle (1.5pt);
\fill (5.45,1.33) circle (1.5pt);
\fill (6.8,0) circle (1.5pt);
\fill (4.2,0) circle (1.5pt);

\draw[thick, mid arrow] (C.west) to (A.south);
\draw[thick, mid arrow] (B.east) to (C.west);
\draw[thick, mid arrow] (C'.west) to (A'.south);
\draw[thick, mid arrow] (B'.east) to (C'.west);
\end{tikzpicture}
\end{center}

The derived graph $Y=X_1 = X(\Z/2\Z, \alpha)$ in this case is a 6-gon.
We draw this graph in two ways for ease of visualization.
\begin{center}
\begin{tikzpicture}[scale=0.75]
\node[inner sep=0pt, label = left:\tiny{$A_1$}] (A1) at (-2,0) {};
\node[inner sep=0pt, label = right:\tiny{$A_2$}] (B3) at (2,0) {}; 
\node[inner sep=0pt, label = above:\tiny{$B_2$}] (B2) at (2*-0.5,2*0.866) {};
\node[inner sep=0pt, label = above:\tiny{$C_2$}] (A2) at (2*0.5,2*0.866) {}; 
\node[inner sep=0pt, label = below:\tiny{$C_1$}] (B1) at (2*-0.5,2*-0.866) {};
\node[inner sep=0pt, label = below:\tiny{$B_1$}] (A3) at (2*0.5,2*-0.866) {}; 

\fill (2,0) circle (1.5pt);
\fill (-2,0) circle (1.5pt);
\fill (2*-0.5,2*0.866) circle (1.5pt);
\fill (2*-0.5,-2*0.866) circle (1.5pt);
\fill (2*0.5,2*0.866) circle (1.5pt);
\fill (2*0.5,-2*0.866) circle (1.5pt);

\draw[thick] (A1) to (B1);
\draw[thick, dashed] (A1) to (B2);
\draw[thick] (B2) to (A2);
\draw[thick] (A2) to (B3);
\draw[thick] (A3) to (B1);
\draw[thick, dashed] (A3) to (B3);
\end{tikzpicture} 
\hspace{1cm}
\begin{tikzpicture}[scale=1]
\node[inner sep=0pt, label = above:\tiny{$A_1$}] (A) at (0,1) {};
\node[inner sep=0pt, label = below:\tiny{$B_1$}] (B) at (-1,0) {}; 
\node[inner sep=0pt, label = below:\tiny{$C_1$}] (C) at (1,0) {}; 
\node[inner sep=0pt, label = above:\tiny{$A_2$}] (A2) at (0,1.5) {};
\node[inner sep=0pt, label = left:\tiny{$B_2$}] (B2) at (-1.5,0) {}; 
\node[inner sep=0pt, label = right:\tiny{$C_2$}] (C2) at (1.5,0) {};   

\fill (0, 1) circle (1.5pt);
\fill (0,1.5) circle (1.5pt);
\fill (1,0) circle (1.5pt);
\fill (-1,0) circle (1.5pt);
\fill (1.5,0) circle (1.5pt);
\fill (-1.5,0) circle (1.5pt);

\draw[thick] (A.south) to (C.west);
\draw[thick, dashed] (A.south) to (B2.east);
\draw[thick] (B2.east) to[bend right=45] (C2.west);
\draw[thick] (A2.south) to (C2.west);
\draw[thick] (C.west) to (B.east);
\draw[thick, dashed] (A2.south) to (B.east);
\end{tikzpicture}
\end{center}
{If we consider the group $G=\Z/4\Z$ instead we obtain a 12-gon}
\begin{center}
\begin{tikzpicture}[scale=1]
\node[inner sep=0pt, shape = circle] (A) at (0,1) {\tiny{$A_1$}};
\node[inner sep=0pt, shape = circle] (B) at (-1,0) {\tiny{$B_1$}}; 
\node[inner sep=0pt, shape = circle] (C) at (1,0) {\tiny{$C_1$}}; 
\node[inner sep=0pt, shape = circle] (A2) at (0,1.5) 
{\tiny{$A_2$}};
\node[inner sep=0pt, shape = circle] (B2) at (-1.5,0) {\tiny{$B_2$}}; 
\node[inner sep=0pt, shape = circle] (C2) at (1.5,0) {\tiny{$C_2$}}; 
\node[inner sep=0pt, shape = circle] (A3) at (0,2) {\tiny{$A_3$}};
\node[inner sep=0pt, shape = circle] (B3) at (-2,0) {\tiny{$B_3$}}; 
\node[inner sep=0pt, shape = circle] (C3) at (2,0) {\tiny{$C_3$}}; 
\node[inner sep=0pt, shape = circle] (A4) at (0,2.5) {\tiny{$A_4$}};
\node[inner sep=0pt, shape = circle] (B4) at (-2.5,0) {\tiny{$B_4$}}; 
\node[inner sep=0pt, shape = circle] (C4) at (2.5,0) {\tiny{$C_4$}}; 

\fill (0, 0.8) circle (1.5pt);
\fill (0,1.3) circle (1.5pt);
\fill (0.775,0.01) circle (1.5pt);
\fill (-0.775,0.01) circle (1.5pt);
\fill (1.3,0.01) circle (1.5pt);
\fill (-1.3,0.01) circle (1.5pt);

\fill (0, 1.8) circle (1.5pt);
\fill (0,2.3) circle (1.5pt);
\fill (1.8,0.01) circle (1.5pt);
\fill (-1.8,0.01) circle (1.5pt);
\fill (2.3,0.01) circle (1.5pt);
\fill (-2.3,0.01) circle (1.5pt);

\draw[thick, dashed] (A.south) to (B2.east);
\draw[thick] (B2.east) to[bend right=45] (C2.west);
\draw[thick] (C2.west) to (A2.south);
\draw[thick, dashed] (A2.south) to (B3.east);
\draw[thick] (B3.east) to[bend right=45] (C3.west);
\draw[thick] (C3.west) to (A3.south);
\draw[thick, dashed] (A3.south) to (B4.east);
\draw[thick] (B4.east) to[bend right=45] (C4.west);
\draw[thick] (C4.west) to (A4.south);
\draw[thick, dashed] (A4.south) to (B.east);
\draw[thick] (B.east) to (C.west);
\draw[thick] (C.west) to (A.south);
\end{tikzpicture}
\end{center}

\end{example}

\begin{example}
\label{another example}
{We consider two different voltage groups $G_1=\Z/2\Z$ and $G_2=\Z/4\Z$.}
Once again there is a unique edge (which we call $e_0$) with non-trivial voltage assignment, i.e., $\alpha(e_0) = \tau$, {where $\tau$ is a generator of $G_1$ and $G_2$, respectively}.
The tower of derived graphs $Y= X_1 = X(\Z/2\Z, \tau)$ and $Y=X_2 = X(\Z/4\Z, \tau)$ is drawn below
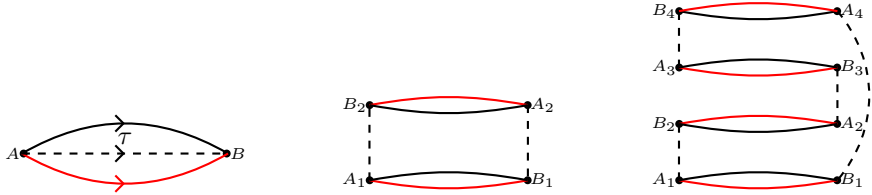
\begin{center}
\begin{tikzpicture}[scale=1]
\node[inner sep=0pt, shape = circle] (A) at (-1.5,0) {\tiny{$A$}};
\node[inner sep=0pt, shape = circle] (B) at (1.5,0) {\tiny{$B$}}; 

\fill (1.35, 0) circle (1.5pt);
\fill (-1.35, 0) circle (1.5pt);

\draw[thick, mid arrow] node[above]{$\tau$} (A.east) to[bend left] (B.west);
\draw[thick, mid arrow, red] (A.east) to[bend right] (B.west);
\draw[thick, mid arrow, dashed] (A.east) to (B.west);
\end{tikzpicture}
\hspace{1cm}
\begin{tikzpicture}[scale=1]
\node[inner sep=0pt, shape = circle] (A) at (-1.25,0) {\tiny{$A_1$}};
\node[inner sep=0pt, shape = circle] (B) at (1.25,0) {\tiny{$B_1$}}; 
\node[inner sep=0pt, shape = circle] (B2) at (-1.25,1) {\tiny{$B_2$}};
\node[inner sep=0pt, shape = circle] (A2) at (1.25,1) {\tiny{$A_2$}}; 

\fill (-1.05, 0) circle (1.5pt);
\fill (1.05,0) circle (1.5pt);
\fill (-1.05,1) circle (1.5pt);
\fill (1.05,1) circle (1.5pt);

\draw[thick, red] (A.east) to[bend right=10] (B.west);
\draw[thick, dashed] (A.east) to (B2.east);
\draw[thick, dashed] (A2.west) to (B.west);
\draw[thick] (A.east) to[bend left=10] (B.west);
\draw[thick, red] (A2.west)[bend right=10] to (B2.east);
\draw[thick] (A2.west)[bend left=10] to (B2.east);
\end{tikzpicture}
\hspace{1cm}
\begin{tikzpicture}[scale=1]
\node[inner sep=0pt, shape = circle] (A) at (-1.25,0) {\tiny{$A_1$}};
\node[inner sep=0pt, shape = circle] (B) at (1.25,0) {\tiny{$B_1$}}; 
\node[inner sep=0pt, shape = circle] (B2) at (-1.25,0.75) {\tiny{$B_2$}};
\node[inner sep=0pt, shape = circle] (A2) at (1.25,0.75) {\tiny{$A_2$}}; 
\node[inner sep=0pt, shape = circle] (A3) at (-1.25,1.5) {\tiny{$A_3$}};
\node[inner sep=0pt, shape = circle] (B3) at (1.25,1.5) {\tiny{$B_3$}}; 
\node[inner sep=0pt, shape = circle] (B4) at (-1.25,2.25) {\tiny{$B_4$}};
\node[inner sep=0pt, shape = circle] (A4) at (1.25,2.25) {\tiny{$A_4$}}; 

\fill (-1.05, 0) circle (1.5pt);
\fill (1.05,0) circle (1.5pt);
\fill (-1.05,0.75) circle (1.5pt);
\fill (1.05,0.75) circle (1.5pt);

\fill (-1.05, 1.5) circle (1.5pt);
\fill (1.05, 1.5) circle (1.5pt);
\fill (-1.05, 2.25) circle (1.5pt);
\fill (1.05, 2.25) circle (1.5pt);

\draw[thick, dashed] (A.east) to (B2.east);
\draw[thick, dashed] (A2.west) to (B3.west);
\draw[thick, dashed] (A3.east) to (B4.east);
\draw[thick, dashed] (A4.west) to[bend left=40] (B.west);
\draw[thick, red] (A.east) to[bend right=10] (B.west);
\draw[thick] (A.east) to[bend left=10] (B.west);
\draw[thick, red] (A2.west) to[bend right=10] (B2.east);
\draw[thick] (A2.west) to[bend left=10] (B2.east);
\draw[thick, red] (A3.east) to[bend right=10] (B3.west);
\draw[thick] (A3.east) to[bend left=10] (B3.west);
\draw[thick] (B4.east) to[bend right=10] (A4.west);
\draw[thick, red] (B4.east) to[bend left=10] (A4.west);
\end{tikzpicture}
\end{center}

{Here we denote} the edges of the base graph $X$ by $e_b$ (for the black edge), $e_r$ (for the red edge), and $e_0$ (for the dashed edge with the non-trivial voltage assignment).
\end{example}

The next result no longer requires the voltage assignment to be unramified.
In other words, the derived graph we obtain is a (ramified) branched cover.

\begin{corollary}
Let $X$ be a finite planar graph and let $Y$ be the derived graph of a single voltage assignment.
Then $Y$ is planar.
\end{corollary}

\begin{proof}
Let $Y^{\unr}$ be the corresponding unramified derived graph.
Note that the immersion $\iota \colon Y^{\unr}\to Y$ is a contraction and that contractions of planar graphs are planar. 
\end{proof}

\begin{example}
\label{ramified version of 3.4}
We once again start with the same graph $X$ as in Example~\ref{unramified derived graph}.  
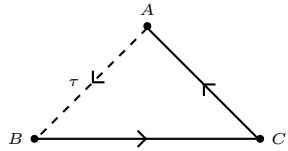
\begin{center}
\begin{tikzpicture}[scale=1]
\node[inner sep=0pt, label = above:\tiny{$A$}] (A) at (0,1.5) {};
\node[inner sep=0pt, label = left:\tiny{$B$}] (B) at (-1.5,0) {}; 
\node[inner sep=0pt, label = right:\tiny{$C$}] (C) at (1.5,0) {}; 
\node[inner sep=0pt, label = left:\tiny{$\tau$}] (T) at (-0.75,0.75) {};

\fill (0,1.5) circle (1.5pt);
\fill (1.5,0) circle (1.5pt);
\fill (-1.5,0) circle (1.5pt);

\draw[thick, mid arrow, dashed] (A.south) to (B.east);
\draw[thick, mid arrow] (C.west) to (A.south);
\draw[thick, mid arrow] (B.east) to (C.west);
\end{tikzpicture}
\end{center}
We no longer require that the vertices are unramified in this extension.
In fact, we assume that the vertices $B,C$ are totally ramified, i.e., the corresponding inertia groups are isomorphic to $G$ at each layer.
The derived graph $Y=X_1=X(\Z/2\Z, \mathcal{I}, \tau)$ is the following
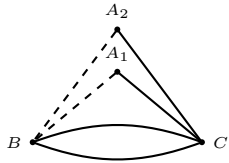
\begin{center}
\begin{tikzpicture}[scale=0.75]
\node[inner sep=0pt, label = above:\tiny{$A_1$}] (A) at (0,1.25) {};
\node[inner sep=0pt, label = left:\tiny{$B$}] (B) at (-1.5,0) {}; 
\node[inner sep=0pt, label = right:\tiny{$C$}] (C) at (1.5,0) {}; 
\node[inner sep=0pt, label = above:\tiny{$A_2$}] (A') at (0,2) {};

\fill (0,1.25) circle (1.5pt);
\fill (0,2) circle (1.5pt);
\fill (1.5,0) circle (1.5pt);
\fill (-1.5,0) circle (1.5pt);

\draw[thick,dashed] (A') to (B);
\draw[thick] (C) to (A');
\draw[thick,dashed] (A) to (B);
\draw[thick] (C) to (A);
\draw[thick] (B) to[bend left=20] (C);
\draw[thick] (B) to[bend right=20] (C);
\end{tikzpicture}
\end{center}
This is easily seen to be a contraction of the hexagon obtained in Example~\ref{unramified derived graph}.
The vertices $B_1$ and $B_2$ (resp. $C_1$ and $C_2$) have been identified.

Now, when $G\simeq (\Z/2^2\Z) = \Z/4\Z$, the derived graph $Y= X(\Z/4\Z, \mathcal{I}, \tau)$ we obtain is the following
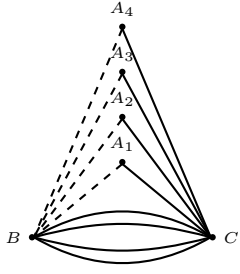
\begin{center}
\begin{tikzpicture}[scale=0.8]
\node[inner sep=0pt, label = above:\tiny{$A_1$}] (A) at (0,1.25) {};
\node[inner sep=0pt, label = left:\tiny{$B$}] (B) at (-1.5,0) {}; 
\node[inner sep=0pt, label = right:\tiny{$C$}] (C) at (1.5,0) {}; 
\node[inner sep=0pt, label = above:\tiny{$A_2$}] (A') at (0,2) {};
\node[inner sep=0pt, label = above:\tiny{$A_3$}] (A'') at (0,2.75) {};
\node[inner sep=0pt, label = above:\tiny{$A_4$}] (A''') at (0,3.5) {};

\fill (0,1.25) circle (1.5pt);
\fill (0,2) circle (1.5pt);
\fill (0,2.75) circle (1.5pt);
\fill (0,3.5) circle (1.5pt);
\fill (1.5,0) circle (1.5pt);
\fill (-1.5,0) circle (1.5pt);

\draw[thick,dashed] (A'.south) to (B.east);
\draw[thick] (C.west) to (A'.south);
\draw[thick,dashed] (A''.south) to (B.east);
\draw[thick] (C.west) to (A''.south);
\draw[thick,dashed] (A'''.south) to (B.east);
\draw[thick] (C.west) to (A'''.south);
\draw[thick,dashed] (A.south) to (B.east);
\draw[thick] (C.west) to (A.south);
\draw[thick] (B.east) to[bend left] (C.west);
\draw[thick] (B.east) to[bend right] (C.west);
\draw[thick] (B.east) to[bend left=15] (C.west);
\draw[thick] (B.east) to[bend right=15] (C.west);
\end{tikzpicture}
\end{center}
This is a contraction of the cyclic 12-gon graph above in Example~\ref{unramified derived graph}.
\end{example}

\section{Towers of Derived Graphs}
\label{sec: towers of derived graphs}

Throughout this section, we assume that $\mathcal{G}\cong \Z_p^d$.
{We write $\gamma_1, \ldots, \gamma_d$ to denote the topological generators of $\cG$.}
By \cite[p.~1336]{dubose-vallieres}, we know that all unramified $\Z_p^d$-towers arise as derived graphs.

\begin{lemma}
\label{groups}
Let $I\subset \mathcal{G}$ be a closed subgroup {and $\sigma\in I$}.
Let $s_n(\sigma)$ denote the order of the image of $\sigma$ in $\mathcal{G}/\mathcal{G}^{p^n}$.
If $s_n(\sigma)\neq 1$, then \[s_{n+1}(\sigma)=ps_n(\sigma).
\]
\end{lemma} 

\begin{proof}
For every $1 \le i \le d$ and $m\geq 0$, there exist integers $1 \leq a_{i,m} \leq  p^m$ such that 
\[\sigma\equiv \prod_{i=1}^d\gamma_i^{a_{i,m}}\pmod{ \mathcal{G}^{p^m}} \textrm{ and } a_{i,m}\equiv a_{i,m+1}\pmod {p^{m}}.
\]
Set $b_m=\min_i(\ord_p(a_{i,m}))$.
By definition $s_m(\sigma)=p^{m-b_m}$.

Let $n$ be a fixed integer satisfying $s_n(\sigma)>1$.
This is equivalent to $b_n<n$.
This inequality, combined with the congruence condition mod ${p^{n}}$ ensures that $b_{n+1}=b_n$; hence $s_{n+1}(\sigma)=ps_n(\sigma)$. 
\end{proof}

{Let $I\subset \mathcal{G}$ be a subgroup of rank $k$. 
The definition in the footnote wouldn't make sense for this?
Also, elsewhere (like Theorem~4.16) we just write $\rk(I_v)$.}
In other words, there are $k$-many topological generators of $I$, say $\sigma_{1},\dots , \sigma_{k}$.
{Let $n\geq 0$ and denote by $I\mathcal{G}^{p^n}$ the smallest subgroup of $\mathcal{G}$ containing $\mathcal{G}^{p^n}$ and $I$. 
Note that $\mathcal{G}^{p^n}$ is a normal subgroup of $\cG$.
Thus, the quotient $I\mathcal{G}^{p^n}/\mathcal{G}^{p^n}$ is well defined.
Fix an integer $n_{{I}}$ such that the $p$-rank\footnote{For an abelian group $G$, we define its $p$-rank to be the $\Z/p\Z$-dimension of $G[p]$.} of $I\mathcal{G}^{p^{n_{I}}}/\mathcal{G}^{p^{n_{I}}}$ is equal to $k$.
Lemma~\ref{groups} now implies that the $p$-rank of $I\mathcal{G}^{p^n}/\mathcal{G}^{p^n}$ is equal to $k$ for all $n\ge n_{{I}}$.}
Thus, given such a group $I\subset \mathcal{G}$ of {rank} $k$ for all $n\ge n_0$ we have 
\[
I\mathcal{G}^{p^n}/\mathcal{G}^{p^n} = \bigoplus_{t=1}^k C_{n,t},
\]
where $C_{n,t}$ is the cyclic subgroup of $\mathcal{G}/\mathcal{G}^{p^n}$ generated by the image of $\sigma_{t}$.

Define $\vec{\mathbf{j}} = (j_1, \ldots, j_k)$ where for each $1 \le t \le k$, the entry $j_t$ can take all (non-negative) integral values between 0 and $s_n(\sigma_{t}) -1$ (both included).
For $n\ge {n_{0}}$, define
\begin{equation}
\label{omega v n defn}
\bomega_{{I},n} = \sum_{\vec{\mathbf{j}}}\prod_{t=1}^k \sigma_{t}^{j_t}\in \Lambda.
\end{equation}

{Intuitively the element $\bomega_{I,n}$ is the sum of all elements in $I\mathcal{G}^{p^n}/\mathcal{G}^{p^n}$.
It therefore can be understood as the norm of the subgroup $I$ in $\mathcal{G}/\mathcal{G}^{p^n}$.
In most cases where the above definition/notation will be used, the subgroup $I$ will be an inertia subgroup $I_v$ of a vertex $v$ in a graph.
In such cases, we will use the notation $\bomega_{v,n}$ instead of $\bomega_{I_v,n}$.}

\begin{definition}
\label{Pv}
Let $X$ be a finite graph with vertices $v_1,\dots,v_s$ and voltage assignment $\alpha$.
Let $X_\infty$ be a branched cover of $X$.
Label the vertices which are unramified in $X_\infty/X$ as $v_1,\dots, v_r$ and the ones which ramify ones as $v_{r+1},\dots ,v_s$.
For each vertex $v$, define
\begin{align*}
P_v & =\deg(v)(v,{I_v})-\sum_{e\in \E^o_{v}(X)}(t(e),\alpha(e)I_{t(e)})\in \Div(X_\infty)
\end{align*}

\noindent $\bullet$
Define
\[
{\Pr}_\Lambda^{\unr}(X_\infty) := \Lambda\textrm{-submodule of } \Div_\Lambda(X_\infty) \textrm{ generated by } \{P_{v_j} \mid 1 \le j \le r\}.
\]

\noindent $\bullet$  
Set
\[
n_0=\displaystyle\max_{r+1 \le j \le s} n_{v_j}
\]
{where $n_{v_j} = n_{I_{v_j}}$ is defined as before for the inertia group $I_{v_j}$ at the vertex $v_j$.}
For each $n\ge n_0$, define
\[
{\Pr}_{n}^{\ram}(X_\infty) = \Lambda\textrm{-submodule of } \Div_\Lambda(X_\infty) \textrm{ generated by } \bomega_{v_j,n}P_{v_j} \textrm{ for } r+1\le j \le s.
\]
\end{definition}

\begin{remark}
For the unramified vertices these are the generators one would expect for the principal divisors.
The factors $\bomega_{v_j,n}$ occur for the ramified vertices due to the immersion map $X_\infty^{\unr}\to X_\infty$ as we will see in the following proof.
\end{remark}


\subsection{Picard Group of Ramified Derived Graphs}

In this section, we prove the $d$-dimensional analogue of \cite[Theorem~5.4]{GV24}.
Let $G$ be any group,  let $Y=X(G, \mathcal{I}, \alpha)$ denote the derived graph of $X$ and write $Y^{\unr} = X(G,\alpha)$.
Let 
\[
\iota \colon Y^{\unr} \longrightarrow Y
\]
be the immersion.
There is a natural surjective group morphism 
\[
\iota_{*} \colon \Div(Y^{\unr}) \longrightarrow \Div(Y).
\]

\begin{theorem} \label{thm:projection}
Let $X_\infty = X(\Z_p^d, \mathcal{I}, \alpha)$ denote the (ramified) derived graph of $X$ and $X_\infty^{\unr} = X(\Z_p^d, \alpha)$ be the unramified derived graph.
Denote the respective layers of the derived graph as $X_n$ and $X_n^{\unr}$.
Define 
\[
N_n={\Pr}_\Lambda^{\unr}(X_\infty) + (\omega_n(T_1),\dots, \omega_n(T_d))\Div_\Lambda(X_\infty) + {\Pr}_n^{\ram}(X_\infty),
\]
where $\omega_n(T_i) = {(T_i+1)}^{p^n} -1$ for $1\le i \le d$.
For $n \gg 0$,
\[
\Div_\Lambda(X_\infty)/N_n\cong \Pic(X_n){\otimes \Z_p}.
\]
\end{theorem}

\begin{proof}
Fix a layer $X_n$ and consider the projection map
\[
\pi_n \colon X_\infty \longrightarrow X_n. 
\]
This induces a surjective map on the divisors, namely
\[
\pi_n \colon \Div_{\Lambda}(X_\infty) \longrightarrow \Div(X_n){\otimes \Z_p}.
\]
To complete the proof of the theorem, it is sufficient to show that $\pi_n^{-1}(\Pr(X_n){\otimes \Z_p}) = N_n$.

{Consider a vertex $v\in V(X)$. 
We fix a pre-image $v'=(v,1)$ of $v$ in $V(X_n^{\unr})$.
Associated to this $v'$, define a principal divisor 
\[
P_v^n=\deg(v)(v,1)-\sum_{e\in \E^o_{v'}(X_n^{\textup{unr}})}t(e) \in \Div(X_n^{\unr}).
\]}
Define
\[
P_{v,n} = \iota_{*}(P^n_v).
\]
A similar definition can be made over $X^{\textup{unr}}_\infty$.
{To simplify notation we write $P_v$ instead of $P_{v,\infty}$ (compare with Definition \ref{Pv}).
Note that $P^\infty_{v}\in \Div(X_\infty^{\unr})$, while $P_v = \iota_*(P_v^{\infty}) \in \Div(X_\infty)$.}
For each vertex $v$ in $X_n$ , we have
\[
\pi_n(P_{v}) =P_{v,n}.
\]
Recall that associated to each ramified vertex $v$, we can associate a{n inertia} group $I_v \subseteq \Z_p^d$ of rank $k = k(v)$.
Let $n\geq n_{0,v}$ where $n_{0,v}$ is a large enough integer such that $\rk_p(I_v\mathcal{G}/\mathcal{G}^{p^{n_{0,v}}})=\rk_p(I_v)$.
Then
\begin{align*}
\pi_n(\bomega_{v,n}P_{v}) & =\sum_{\vec{\mathbf{j}}}\prod_{t=1}^k\sigma_{v,t}^{j_t}P_{v,n}=\sum_{\sigma \in I_v\mathcal{G}^{p^n}/ \mathcal{G}^{p^n}}\sigma P_{v,n}
\end{align*}
and \cite[Corollary 4.7]{GV24} implies that $\pi_n(N_n)=\Pr(X_n){\otimes \Z_p}$. Therefore 
\[
N_n\subset \pi_n^{-1}(\Pr(X_n){\otimes \Z_p}) \textrm{ and } \pi_n^{-1}(\Pr(X_n){\otimes \Z_p})=N_n+\ker(\pi_n).
\]

Write $\textup{Aug}(I_{v_l})$ to denote the augmentation ideal of $I_{v_l}$.
The arguments in \cite[Section 5]{GV24} now show
\[\Div_\Lambda(X_\infty)=\Lambda^r\oplus \bigoplus_{l=r+1}^s \Lambda/\textup{Aug}(I_{v_l}).
\]
Taking appropriate projections,
\[
\Div(X_n)\cong \Z_p[\mathcal{G}/\mathcal{G}^{p^n}]^r\oplus \bigoplus_{l=r+1}^s \Z_p[\mathcal{G}/I_{v_l}\mathcal{G}^{p^n}].
\]
Therefore for $n\gg 0$,
\[
\ker(\pi_n)=(\omega_n(T_1),\dots,\omega_n(T_d))\Div_\Lambda(X_\infty).
\]
In particular, $\ker(\pi_n)\subset N_n$ and $\pi_n^{-1}(\Pr(X_n){\otimes \Z_p})=N_n$. 

\end{proof}

Even though the module $\Pr_n^{\ram}(X_\infty)$ is necessary to describe $\Pic(X_n){\otimes \Z_p}$ as a quotient of $\Div_\Lambda(X_\infty)$ it does not play a role when we take protective limits as the following corollary points out.
{In view of the notation introduced previously,
\[
\Pic_\Lambda(X_\infty) = \Pic(X_\infty) \otimes_{\Z_p[\mathcal{G}]} {\Lambda}.
\]}

\begin{corollary}
\label{cor:ideals}
With notation as above,
\[
\Pic_\Lambda(X_\infty)\cong \Div_\Lambda(X_\infty)/{\Pr}^{\unr}_\Lambda(X_\infty).
\]
\end{corollary}

\begin{proof}
This is analogous to \cite[Corollary 5.5]{GV24} using Theorem~\ref{thm:projection} instead of \cite[Theorem 5.4]{GV24}.
\end{proof}

{Fix $v\in V(X)$.}
Let $w_{v,\infty}^{\unr}$ denote the vertex $(v, \mathbf{1}_{\cG})$ in $X_\infty^{\unr} = X(\cG,\alpha)$.
Let $\Delta$ be the operator on $\Div_\Lambda(X_\infty^{\unr})$ given by \[
\Delta w^{\unr}_{v,\infty}=
\begin{cases}
P_v^{\unr} &\textrm{ if } v \textrm{ is unramified in } X_\infty/X\\
(\sigma_{v,1}-1)w_{v,\infty}^{\unr} &\textrm{ if }
v \textrm{ is ramified in } X_\infty/X \textrm{ and } \rk(I_{v})=1\\
w_{v,\infty}^{\unr} &\textrm{ if }
v \textrm{ is ramified in } X_\infty/X \textrm{ and } \rk(I_{v})>1,
\end{cases}
\]
where 
\[
P^{\unr}_v =\deg(v)(v,1)-\sum_{e\in \E^o_{v}(X)}(t(e),\alpha(e))\in \Div(X^{\unr}_\infty).
\]
In previous references, the operator $P_v^{\unr}$ is often written as $\mathcal{L}_{\Lambda}^{\unr}(-)$ and is called the Laplacian operator.
This notation is consistent with how $P_v$ was expressed in Definition~\ref{Pv}; only now $I_{t(e)}$ is trivial.
{Note that this divisor was denoted by $P_v^{\infty}$ before.}

\begin{remark}
If $X$ is a finite connected graph and $X_\infty =X_\infty^{\unr}$ is an unramified $\Z_p^d$-tower of $X$, this operator $\Delta$ was already studied in \cite{Kleine-Mueller4} and was denoted by $\Delta_\infty$.
\end{remark}
 
\begin{theorem}
\label{thM.char-ideals}
Let $X_\infty = X(\Z_p^d, \mathcal{I}, \alpha)$ denote the (ramified) derived graph of $X$.
With this notation
\[
\Char_\Lambda(\Pic_\Lambda(X_\infty))=(\det(\Delta)).
\]
\end{theorem}

\begin{proof}
Write {as before} $X_\infty^{\unr} = X(\Z_p^d, \alpha)$ to denote the unramified derived graph.
Define\footnote{We want to emphasize that $\Pr^{\unr}_\Lambda(X_\infty)$ was defined as a $\Lambda$-submodule of $\Div_\Lambda(X_\infty)$ generated by $\{P_{v_l} \mid 1 \le l \le r\}$ in Definition~\ref{Pv} which is different from $\Pr'_\Lambda(X_\infty^{{\unr}})$.
The two definitions coincide if $X_\infty = X_\infty^{\unr}$.} 
\[
{\Pr}'_\Lambda(X_\infty^{{\unr}}) = \Lambda\textrm{-submodule of } \Div_{\Lambda}(X_\infty^{\unr}) \textrm{ generated by } \{P_{v_l}^{\unr}\mid 1\le l\le r\}.
\]
Note that the quotient $\Div_\Lambda(X_\infty^{\unr})/\Pr'_\Lambda(X_\infty^{\unr})$ has $\Lambda$-rank $s-r$.
Also, recall the surjection
\[
\iota_* : \Div_{\Lambda}(X_\infty^{\unr}) \longrightarrow \Div_{\Lambda}(X_\infty).
\]
Under this map, the pre-image of $\Pr_{\Lambda}^{\unr}(X_\infty)$ is precisely $\ker(\iota_*)+\textup{Pr}'_\Lambda(X_\infty^{\unr})$.
Now, it follows immediately from Corollary~\ref{cor:ideals} that
\[
\Pic_\Lambda(X_\infty) \cong \Div_\Lambda(X_\infty)/\textup{Pr}^{\unr}_\Lambda(X_\infty) =\Div_\Lambda(X_\infty^{\unr})/\ker(\iota_*)+\textup{Pr}'_\Lambda(X_\infty^{\unr}).
\]
Thus it suffices to compute the characteristic ideal of $\Div_\Lambda(X_\infty^{\unr})/\ker(\iota_*)+\textup{Pr}'_\Lambda(X_\infty^{\unr})$.
Consider the exact sequence
\[
0\to \Image(\Delta)/\ker(\iota_*)+\textup{Pr}'_\Lambda(X_\infty^{\unr})\to \Div_\Lambda(X_\infty^{\unr})/\ker(\iota_*)+\textup{Pr}'_\Lambda(X_\infty^{\unr})\to \Div_\Lambda(X^{\unr}_{\infty})/\Image(\Delta)\to 0.
\]
To prove the theorem we show that the first term is pseudo-null.
{By definition, $\Image(\Delta)$ is generated by 
\begin{align*}
&\{w^\textup{unr}_{v,\infty}\mid v\in \{v_{r+1},\dots,v_s\}, \  \rk(I_v)\ge 2\}\\
&\cup \{(\sigma_{v,1}-1)w_{v,\infty}^{\unr}\mid v\in \{v_{r+1},\dots,v_s\}, \  \rk(I_v)=1\}\cup \{P_v^{\unr}\mid\textrm{ if } v \textrm{ is an unramified vertex in } X_\infty/X\}.
\end{align*}
The last two sets are contained in $\ker(\iota_*)+\Pr'_\Lambda(X_\infty^{\unr})$. Thus the quotient $\Image(\Delta)/\ker(\iota_*)+\textup{Pr}'_\Lambda(X_\infty^{\unr})$}
 is generated by 
\[
\{w^\textup{unr}_{v,\infty}\mid v\in \{v_{r+1},\dots,v_s\}, \  \rk(I_v)\ge 2\}
\]
Each $w^\textup{unr}_{v,\infty}$ is annihilated by $\textup{Aug}(I_v)$.
As $\rk(I_v)\ge 2$, the {ideal $\textup{Aug}(I_v)$ has height at least $2$.
In particular,} the module $\Lambda/\textup{Aug}(I_v)$ is pseudo-null and the result follows; {see p.~\pageref{pseudooo} for the definition of pseudo-null modules}.
\end{proof}

\begin{remark}
If $X_\infty = X_\infty^{\unr}$, the above theorem asserts that
\[
\Char_{\Lambda}(\Pic_{\Lambda}(X_\infty^{\unr})) = (\det(\Delta)).
\]
On the other hand, it was proven in \cite[Theorem~5.2]{Kleine-Mueller4} that if $X_\infty/X$ is an unramified $\Z_p^d$-extension of $X$ (obtained as a derived graph of $X$) then
\[
(\det(\Delta)) = \begin{cases}
    \Char_{\Lambda}(\Jac_{\Lambda}(X_\infty)) & \textrm{ if } d\ge 2\\
    (T)\Char_{\Lambda}(\Jac_{\Lambda}(X_\infty)) & \textrm{ if } d=1.
\end{cases}
\]
Therefore, we conclude
\[
\Char_{\Lambda}(\Pic_{\Lambda}(X_\infty)) = \begin{cases}
    \Char_{\Lambda}(\Jac_{\Lambda}(X_\infty)) & \textrm{ if } d\ge 2\\
    (T)\Char_{\Lambda}(\Jac_{\Lambda}(X_\infty)) & \textrm{ if } d=1.
\end{cases}
\]
Compare also with \cite[Remark~5.4]{Kleine-Mueller4}.
\end{remark}

\begin{example}
We give an example of an unramified cover of $X$ that can also be found in \cite{Kleine-Mueller4}.
Here the dashed line between $A$ and $B$ in the graph at the bottom layer has non-trivial voltage assignment, say $\tau$.

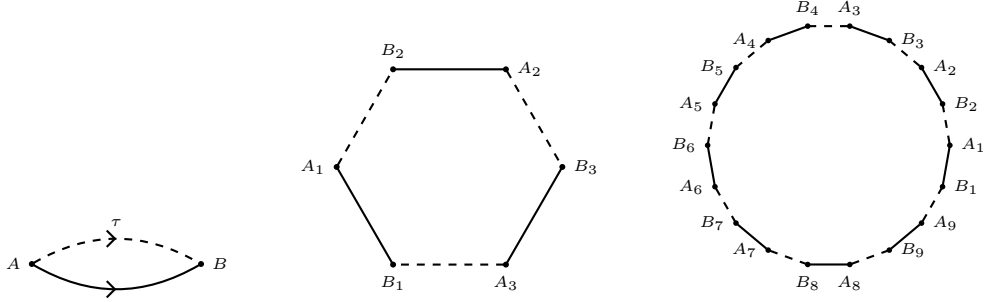
\begin{center}
\begin{tikzpicture}[scale=0.75]
\node[inner sep=0pt, label = left:{\tiny{$A$}}] (A) at (-1.5,0) {};
\node[inner sep=0pt, label = right:{\tiny{$B$}}] (B) at (1.5,0) {}; 
\node[inner sep=0pt, label = below:{\tiny{$\tau$}}] (T) at (0,1) {}; 

\fill (1.5, 0) circle (1.5pt);
\fill (-1.5, 0) circle (1.5pt);

\draw[thick, mid arrow, dashed] (A) to[bend left] (B);
\draw[thick, mid arrow] (A) to[bend right] (B);
\end{tikzpicture}
\hspace{0.5cm}
\begin{tikzpicture}[scale=0.75]
\node[inner sep=0pt, label = left:{\tiny{$A_1$}}] (A1) at (-2,0) {};
\node[inner sep=0pt, label = right:{\tiny{$B_3$}}] (B3) at (2,0) {}; 
\node[inner sep=0pt, label = above:{\tiny{$B_2$}}] (B2) at (2*-0.5,2*0.866) {};
\node[inner sep=0pt, label = right:{\tiny{$A_2$}}] (A2) at (2*0.5,2*0.866) {}; 
\node[inner sep=0pt, label = below:{\tiny{$B_1$}}] (B1) at (2*-0.5,2*-0.866) {};
\node[inner sep=0pt, label = below:{\tiny{$A_3$}}] (A3) at (2*0.5,2*-0.866) {}; 

\fill (-2,0) circle (1.5pt);
\fill (2,0) circle (1.5pt);
\fill (2*-0.5,2*0.866) circle (1.5pt);
\fill (2*-0.5,-2*0.866) circle (1.5pt);
\fill (2*0.5,2*0.866) circle (1.5pt);
\fill (2*0.5,-2*0.866) circle (1.5pt);

\draw[thick] (A1) to (B1);
\draw[thick, dashed] (A1) to (B2);
\draw[thick] (B2) to (A2);
\draw[thick, dashed] (A2) to (B3);
\draw[thick, dashed] (A3) to (B1);
\draw[thick] (A3) to (B3);
\end{tikzpicture}
\hspace{0.5cm}
\begin{tikzpicture}[scale=0.7]
\node[inner sep=0pt, label = left:{\tiny{$B_6$}}] (B6) at (-2.3*1,0){};
\node[inner sep=0pt, label = left:{\tiny{$A_5$}}] (A5) at (2.3*-0.93969,2.3*0.34202){};
\node[inner sep=0pt, label = left:{\tiny{$A_4$}}] (A4) at (2.3*-0.5,2.3*0.866){};
\node[inner sep=0pt, label = above:{\tiny{$A_3$}}] (A3) at (2.3*0.173648,2.3*0.9848){};
\node[inner sep=0pt, label = right:{\tiny{$A_2$}}] (A2) at (2.3*0.766,2.3*0.64278){};
\node[inner sep=0pt, label = right:{\tiny{$A_1$}}] (A1) at (2.3*1,0){};
\node[inner sep=0pt, label = right:{\tiny{$A_9$}}] (A9) at (2.3*0.766,2.3*-0.64278){};
\node[inner sep=0pt, label = below:{\tiny{$A_8$}}] (A8) at (2.3*0.173648,2.3*-0.9848){};
\node[inner sep=0pt, label = left:{\tiny{$A_7$}}] (A7) at (2.3*-0.5,2.3*-0.866){};
\node[inner sep=0pt, label = left:{\tiny{$A_6$}}] (A6) at (2.3*-0.93969,2.3*-0.34202){};
\node[inner sep=0pt, label = left:{\tiny{$B_5$}}] (B5) at (2.3*-0.766,2.3*0.64278){};
\node[inner sep=0pt, label = left:{\tiny{$B_7$}}] (B7) at (2.3*-0.766,2.3*-0.64278){};
\node[inner sep=0pt, label = above:{\tiny{$B_4$}}] (B4) at (2.3*-0.173648,2.3*0.9848){};
\node[inner sep=0pt, label = below:{\tiny{$B_8$}}] (B8) at (2.3*-0.173648,2.3*-0.9848){};
\node[inner sep=0pt, label = right:{\tiny{$B_3$}}] (B3) at (2.3*0.5,2.3*0.866){};
\node[inner sep=0pt, label = right:{\tiny{$B_2$}}] (B2) at (2.3*0.93969,2.3*0.34202){};
\node[inner sep=0pt, label = right:{\tiny{$B_1$}}] (B1) at (2.3*0.93969,2.3*-0.34202){};
\node[inner sep=0pt, label = right:{\tiny{$B_9$}}] (B9) at (2.3*0.5,2.3*-0.866){};

\fill (-2.3*1,0) circle (1.5pt);
\fill (2.3*-0.93969,2.3*0.34202) circle (1.5pt);
\fill (2.3*-0.5,2.3*0.866) circle (1.5pt);
\fill (2.3*0.173648,2.3*0.9848) circle (1.5pt);
\fill (2.3*0.766,2.3*0.64278) circle (1.5pt);
\fill (2.3*1,0) circle (1.5pt);
\fill (2.3*0.766,2.3*-0.64278) circle (1.5pt);
\fill (2.3*0.173648,2.3*-0.9848) circle (1.5pt);
\fill (2.3*-0.5,2.3*-0.866) circle (1.5pt);
\fill (2.3*-0.93969,2.3*-0.34202) circle (1.5pt);

\fill (2.3*-0.766,2.3*0.64278) circle (1.5pt);
\fill (2.3*-0.766,2.3*-0.64278) circle (1.5pt);
\fill (2.3*-0.173648,2.3*0.9848) circle (1.5pt);
\fill (2.3*-0.173648,2.3*-0.9848) circle (1.5pt);

\fill (2.3*0.5,2.3*0.866) circle (1.5pt);
\fill (2.3*0.5,2.3*-0.866) circle (1.5pt);
\fill (2.3*0.93969,2.3*0.34202) circle (1.5pt);
\fill (2.3*0.93969,2.3*-0.34202) circle (1.5pt);

\draw[thick] (A7) to (B7);
\draw[thick] (A5) to (B5);
\draw[thick] (A6) to (B6);
\draw[thick, dashed] (A6) to (B7);
\draw[thick, dashed] (A5) to (B6);
\draw[thick, dashed] (A3) to (B4);
\draw[thick] (A8) to (B8);
\draw[thick, dashed] (A4) to (B5);
\draw[thick, dashed] (A7) to (B8);
\draw[thick] (A4) to (B4);
\draw[thick, dashed] (A2) to (B3);
\draw[thick] (A3) to (B3);
\draw[thick] (A9) to (B9);
\draw[thick, dashed] (A9) to (B1);
\draw[thick] (A2) to (B2);
\draw[thick, dashed] (A8) to (B9);
\draw[thick] (A1) to (B1);
\draw[thick, dashed] (A1) to (B2);
\end{tikzpicture} 
\end{center}

For convenience of the reader, we draw another way to interpret the 18-gon in the third diagram above:

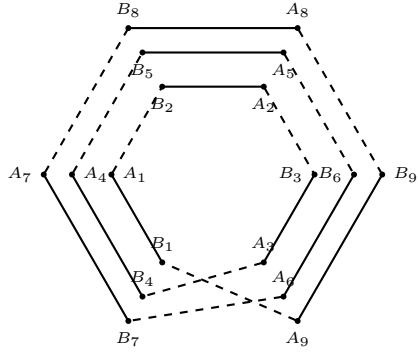
\begin{center}
\begin{tikzpicture}[scale=0.75]
\node[inner sep=0pt, label = right:\tiny{$A_1$}] (A1) at (-1.8,0) {};
\node[inner sep=0pt, label = left:\tiny{$B_3$}] (B3) at (1.8,0) {}; 
\node[inner sep=0pt, label = below:\tiny{$B_2$}] (B2) at (1.8*-0.5,1.8*0.866) {};
\node[inner sep=0pt, label = below:\tiny{$A_2$}] (A2) at (1.8*0.5,1.8*0.866) {}; 
\node[inner sep=0pt, label = above:\tiny{$B_1$}] (B1) at (1.8*-0.5,1.8*-0.866) {};
\node[inner sep=0pt, label = above:\tiny{$A_3$}] (A3) at (1.8*0.5,1.8*-0.866) {}; 
\node[inner sep=0pt, label = left:\tiny{$A_7$}] (A7) at (-3,0) {};
\node[inner sep=0pt, label = right:\tiny{$B_9$}] (B9) at (3,0) {}; 
\node[inner sep=0pt, label = above:\tiny{$B_8$}] (B8) at (3*-0.5,3*0.866) {};
\node[inner sep=0pt, label = above:\tiny{$A_8$}] (A8) at (3*0.5,3*0.866) {}; 
\node[inner sep=0pt, label = below:\tiny{$B_7$}] (B7) at (3*-0.5,3*-0.866) {};
\node[inner sep=0pt, label = below:\tiny{$A_9$}] (A9) at (3*0.5,3*-0.866) {}; 
\node[inner sep=0pt, label = right:\tiny{$A_4$}] (A4) at (-2.5,0) {};
\node[inner sep=0pt, label = left:\tiny{$B_6$}] (B6) at (2.5,0) {}; 
\node[inner sep=0pt, label = below:\tiny{$B_5$}] (B5) at (2.5*-0.5,2.5*0.866) {};
\node[inner sep=0pt, label = below:\tiny{$A_5$}] (A5) at (2.5*0.5,2.5*0.866) {}; 
\node[inner sep=0pt, label = above:\tiny{$B_4$}] (B4) at (2.5*-0.5,2.5*-0.866) {};
\node[inner sep=0pt, label = above:\tiny{$A_6$}] (A6) at (2.5*0.5,2.5*-0.866) {}; 

\fill (-1.8,0) circle (1.5pt);
\fill (1.8,0) circle (1.5pt);
\fill (1.8*-0.5,1.8*0.866) circle (1.5pt);
\fill (1.8*0.5,1.8*0.866) circle (1.5pt);
\fill (1.8*-0.5,1.8*-0.866) circle (1.5pt);
\fill (-1.8*-0.5,1.8*-0.866) circle (1.5pt);

\fill (3,0) circle (1.5pt);
\fill (-3,0) circle (1.5pt);
\fill (3*-0.5,3*0.866) circle (1.5pt);
\fill (3*-0.5,-3*0.866) circle (1.5pt);
\fill (3*0.5,3*0.866) circle (1.5pt);
\fill (3*0.5,-3*0.866) circle (1.5pt);

\fill (2.5,0) circle (1.5pt);
\fill (-2.5,0) circle (1.5pt);
\fill (2.5*-0.5,2.5*0.866) circle (1.5pt);
\fill (2.5*-0.5,-2.5*0.866) circle (1.5pt);
\fill (2.5*0.5,2.5*0.866) circle (1.5pt);
\fill (2.5*0.5,-2.5*0.866) circle (1.5pt);

\draw[thick] (A1) to (B1);
\draw[thick, dashed] (A1) to (B2);
\draw[thick] (B2) to (A2);
\draw[thick, dashed] (A2) to (B3);
\draw[thick, dashed] (A9) to (B1);
\draw[thick] (A3) to (B3);

\draw[thick] (A4) to (B4);
\draw[thick, dashed] (A4) to (B5);
\draw[thick] (B5) to (A5);
\draw[thick, dashed] (A5) to (B6);
\draw[thick, dashed] (A3) to (B4);
\draw[thick] (A6) to (B6);

\draw[thick] (A7) to (B7);
\draw[thick, dashed] (A7) to (B8);
\draw[thick] (B8) to (A8);
\draw[thick, dashed] (A8) to (B9);
\draw[thick, dashed] (A6) to (B7);
\draw[thick] (A9) to (B9);
\end{tikzpicture}
\end{center}
The characteristic polynomial of $\Pic_{\Lambda}(X_\infty)$ is $T^2$ for this unramified tower.
\end{example}

For ramified branched covers $X_\infty$, it has been pointed out \cite[Section~6]{GV24} that $\det(\Delta) = \det(D-B)$ where $D, B$ were matrices introduced in Definition~\ref{matrix defn}.
We now provide examples of the same.

\begin{example}
Reconsider the branched covering from Example~\ref{ramified version of 3.4}.
We calculate the matrices
\[
D = \begin{bmatrix}
    2 & 0 & 0\\
    0 & 0 & 0\\
    0 & 0 & 0
\end{bmatrix} \textrm{ and } B = \begin{bmatrix}
    0 & 0 & 0 \\
    T+1 & -T & 0 \\
    1 & 0 & -T
\end{bmatrix}.
\]
It follows that
\[
\Char_\Lambda(\Pic_\Lambda(X_\infty)) = \det(D-B) = \det\begin{bmatrix}
    2 & 0 & 0 \\
    -(1+T) & T & 0 \\
    -1& 0 & T
\end{bmatrix} = 2T^2.
\]
\end{example}

\begin{example}
\label{ex 4.9}
In this example $G=(\Z/3\Z, +)$ for each subsequent layer and the vertex $R\in V(X)$ is totally ramified whereas the vertex $U\in V(X)$ is unramified.
The dashed edge has non-trivial voltage assignment.

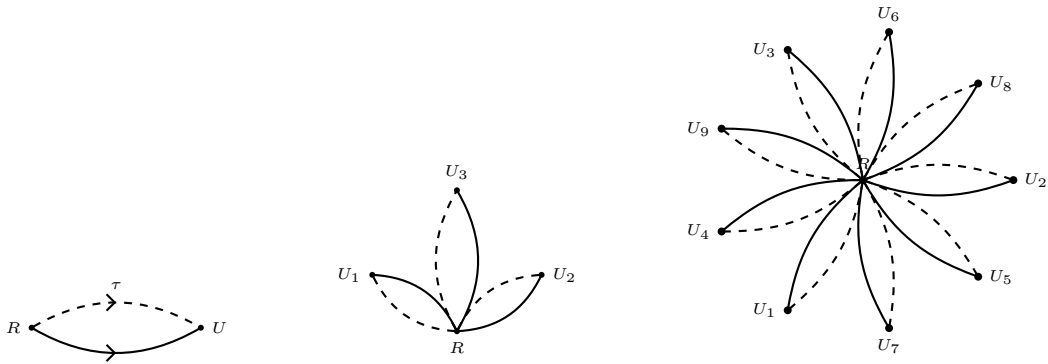
\begin{center}
\begin{tikzpicture}[scale=0.75]
\node[inner sep=0pt, label = left:{\tiny{$R$}}] (A) at (-1.5,0) {};
\node[inner sep=0pt, label = right:{\tiny{$U$}}] (B) at (1.5,0) {}; 
\node[inner sep=0pt, label = below:{\tiny{$\tau$}}] (T) at (0,1) {}; 

\fill (1.5, 0) circle (1.5pt);
\fill (-1.5, 0) circle (1.5pt);

\draw[thick, mid arrow, dashed] (A) to[bend left] (B);
\draw[thick, mid arrow] (A) to[bend right] (B);
\end{tikzpicture}
\hspace{1cm}
\begin{tikzpicture}[scale=0.75]
\node[inner sep=0pt, label = below:{\tiny{$R$}}] (A) at (0,0) {};
\node[inner sep=0pt, label = left:{\tiny{$U_1$}}] (B) at (-1.5,1) {}; 
\node[inner sep=0pt, label = right:{\tiny{$U_2$}}] (C) at (1.5,1) {}; 
\node[inner sep=0pt, label = above:{\tiny{$U_3$}}] (D) at (0,2.5) {}; 

\fill (0, 0) circle (1.5pt);
\fill (0, 2.5) circle (1.5pt);
\fill (-1.5, 1) circle (1.5pt);
\fill (1.5, 1) circle (1.5pt);

\draw[thick,dashed] (A) to[bend left] (B);
\draw[thick] (A) to[bend right] (B);
\draw[thick, dashed] (A) to[bend left] (D);
\draw[thick] (A) to[bend right] (D);
\draw[thick, dashed] (A) to[bend left] (C);
\draw[thick] (A) to[bend right] (C);
\end{tikzpicture}
\hspace{1cm}
\begin{tikzpicture}[scale=1]
\node[inner sep=0pt, label = above:{\tiny{$R$}}] (A) at (0,0) {};
\node[inner sep=0pt, label = right:{\tiny{$U_2$}}] (C) at (2,0) {}; 
\node[inner sep=0pt, label = right:{\tiny{$U_8$}}] (C') at (2*0.766,2*0.64278) {}; 
\node[inner sep=0pt, label = right:{\tiny{$U_5$}}] (C'') at (2*0.766,2*-0.64278) {}; 
\node[inner sep=0pt, label = above:{\tiny{$U_6$}}] (B) at (2*0.173648,2*0.9848) {}; 
\node[inner sep=0pt, label = left:{\tiny{$U_3$}}] (B') at (2*-0.5,2*0.866) {}; 
\node[inner sep=0pt, label = left:{\tiny{$U_9$}}] (B'') at (2*-0.93969,2*0.34202) {}; 
\node[inner sep=0pt, label = left:{\tiny{$U_4$}}] (D) at (2*-0.93969,2*-0.34202) {}; 
\node[inner sep=0pt, label = left:{\tiny{$U_1$}}] (D') at (2*-0.5,2*-0.866) {}; 
\node[inner sep=0pt, label = below:{\tiny{$U_7$}}] (D'') at (2*0.173648,2*-0.9848) {}; 

\fill (2,0) circle (1.5pt);
\fill (2*0.766,2*0.64278) circle (1.5pt);
\fill (2*0.766,2*-0.64278) circle (1.5pt);
\fill (0,0) circle (1.5pt);
\fill (2*0.173648,2*0.9848) circle (1.5pt);
\fill (2*-0.5,2*0.866) circle (1.5pt);
\fill (2*-0.93969,2*0.34202) circle (1.5pt);
\fill (2*-0.93969,2*-0.34202) circle (1.5pt);
\fill (2*-0.5,2*-0.866) circle (1.5pt);
\fill (2*0.173648,2*-0.9848) circle (1.5pt);

\draw[thick,dashed] (A) to[bend left=20] (C');
\draw[thick] (A) to[bend right=20] (C');
\draw[thick,dashed] (A) to[bend left=20] (D');
\draw[thick] (A) to[bend right=20] (D');
\draw[thick,dashed] (A) to[bend left=20] (D'');
\draw[thick] (A) to[bend right=20] (D'');
\draw[thick,dashed] (A) to[bend left=20] (D);
\draw[thick] (A) to[bend right=20] (D);
\draw[thick,dashed] (A) to[bend left=20] (B');
\draw[thick] (A) to[bend right=20] (B');
\draw[thick,dashed] (A) to[bend left=20] (C'');
\draw[thick] (A) to[bend right=20] (C'');
\draw[thick,dashed] (A) to[bend left=20] (B'');
\draw[thick] (A) to[bend right=20] (B'');
\draw[thick, dashed] (A) to[bend left=20] (C);
\draw[thick] (A) to[bend right=20] (C);
\draw[thick, dashed] (A) to[bend left=20] (B);
\draw[thick] (A) to[bend right=20] (B);
\end{tikzpicture}
\end{center}
We calculate the two matrices $D$ and $B$ using Definition~\ref{matrix defn},
\[
D = \begin{bmatrix}
    2 & 0 \\
    0 & 0
\end{bmatrix} \textrm{ and } B = \begin{bmatrix}
    0 & 0 \\
    \tau^{-1}+1 & 1-(1+T)^{p^0}
\end{bmatrix} = \begin{bmatrix}
    0 & 0 \\
    \tau^{-1}+1 & -T
\end{bmatrix}
\]
It follows that the characteristic ideal
\[
\Char_{\Lambda}(\Pic_{\Lambda}(X_\infty)) = \det(D-B) = \det\begin{bmatrix}
    2 & 0 \\
    -\tau^{-1}-1 & T
\end{bmatrix} = 2T.
\] 
\end{example}

This example can be generalized in the following manner:
\begin{example}
Let $X$ be a cycle graph with $m$ vertices.
Let $\alpha$ be a single voltage assignment with image in $\Z_p$.
Let $I_{v}=\Z_p$ for exactly one vertex and set $I_{v'}=1$ for all $v'\neq v$.
Let $X_\infty=X(\Z_p, \mathcal{I}, \alpha)$ and let $X_n$ be the intermediate covers.
Assume that the image of $\alpha$ generates $\Z_p$ topologically.
Then 
\[
\ord_p(\vert \Jac(X_n)\vert)=\ord_p(m){p^n}.
\]
Indeed, $X_n$ looks similar to the flower graph above. 
Instead of leaves consisting of one vertex and two edges, the leaves have $(m-1)$ vertices and $m$ edges.
To count the spanning trees we need to choose a spanning tree for each leaf, i.e. we have to choose one edge that we "delete".
That leaves $m$ choices per leaf.
As there are $p^n$ leaves the claim follow. 

In the following picture we choose $p=3$ and draw the base graph and the first layer.

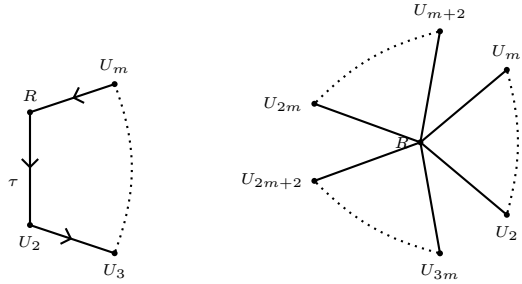
\begin{center}
\begin{tikzpicture}[scale=0.75]
\node[inner sep=0pt, label = above:\tiny{$R$}] (A) at (0,2) {};
\node[inner sep=0pt, label = below:\tiny{$U_2$}] (B) at (0,0) {}; 
\node[inner sep=0pt, label = below:\tiny{$U_3$}] (C) at (1.5,-0.5) {}; 
\node[inner sep=0pt, label = above:\tiny{$U_m$}] (D) at (1.5,2.5) {}; 
\node[inner sep=0pt, label = left:\tiny{$\tau$}] (T) at (-0,0.75) {};

\fill (0,2) circle (1.5pt);
\fill (0,0) circle (1.5pt);
\fill (1.5,-0.5) circle (1.5pt);
\fill (1.5, 2.5) circle (1.5pt);

\draw[thick, mid arrow] (A) to (B);
\draw[thick, mid arrow] (D) to (A);
\draw[thick, mid arrow] (B) to (C);
\draw[thick, dotted] (C) to[bend right=20] (D);
\end{tikzpicture}
\hspace{1cm}
\begin{tikzpicture}[scale=0.75]
\node[inner sep=0pt, label = left:\tiny{$R$}] (A) at (0,0) {};
\node[inner sep=0pt, label = above:\tiny{$U_m$}] (C') at (2*0.766,2*0.64278) {}; 
\node[inner sep=0pt, label = below:\tiny{$U_2$}] (C'') at (2*0.766,2*-0.64278) {}; 
\node[inner sep=0pt, label = above:\tiny{$U_{m+2}$}] (B) at (2*0.173648,2*0.9848) {}; 
\node[inner sep=0pt, label = left:\tiny{$U_{2m}$}] (B'') at (2*-0.93969,2*0.34202) {}; 
\node[inner sep=0pt, label = left:\tiny{$U_{2m+2}$}] (D) at (2*-0.93969,2*-0.34202) {}; 
\node[inner sep=0pt, label = below:\tiny{$U_{3m}$}] (D'') at (2*0.173648,2*-0.9848) {}; 

\fill (0,0) circle (1.5pt);
\fill (2*0.766,2*0.64278) circle (1.5pt);
\fill (2*0.766,2*-0.64278) circle (1.5pt);
\fill (2*0.173648,2*0.9848) circle (1.5pt);
\fill (2*-0.93969,2*0.34202) circle (1.5pt);
\fill (2*-0.93969,2*-0.34202) circle (1.5pt);
\fill (2*0.173648,2*-0.9848) circle (1.5pt);

\draw[thick] (A) to[] (C');
\draw[thick] (A) to[] (C'');
\draw[thick, dotted] (C') to[bend left = 15] (C'');
\draw[thick] (A) to[] (B);
\draw[thick] (A) to[] (B'');
\draw[thick, dotted] (B) to[bend right = 15] (B'');
\draw[thick] (A) to[] (D);
\draw[thick] (A) to[] (D'');
\draw[thick, dotted] (D) to[bend right = 15] (D'');
\end{tikzpicture}
\end{center}

We calculate the two matrices $D$ and $B(T)$ using Definition~\ref{matrix defn},
\[
D= 
\begin{bmatrix}
    2 & 0 & 0 & \cdots & 0 \\
    0 & 2 & 0 & \cdots & 0 \\
    0 & 0 & 2 & \cdots & 0 \\
    \vdots & \vdots & \ddots & \ddots & \vdots\\
    0 & 0 & 0 &  \cdots & 0
\end{bmatrix}_{m\times m} \textrm{ and } B(T) = \begin{bmatrix}
0 & 1 & 0 & \ldots & 0 \\
1 & 0 & 1 & \ldots & 0 \\
0 & 1 & 0 & \ldots & 0 \\
\vdots & \vdots & \ddots & \ddots & \vdots\\
\tau^{-1}  & 0 & 0 &  & -T
\end{bmatrix}_{m\times m}
\]
Therefore,
\[
\Char_{\Lambda}(\Pic_{\Lambda}(X_\infty)) = \det(D-B) = mT
\]
\end{example}

\begin{example}
In the subsequent example we consider the case when $\mathcal{G} = \Z_3^2$, i.e. the first layer is a $\Z/3\Z\times \Z/3\Z$ covering. We assume that the vertex $U$ is unramified and that the vertex $R$ is totally ramified.
{Recall that the voltage assignment $\alpha: \E(X) \to \Z_3 \times \Z_3$.
Let us denote the generators of $\cG$ by $\sigma, \tau$.
The voltage assignment for the dashed edge is prescribed to be $\tau$ and that for the solid edge is $\sigma$.}
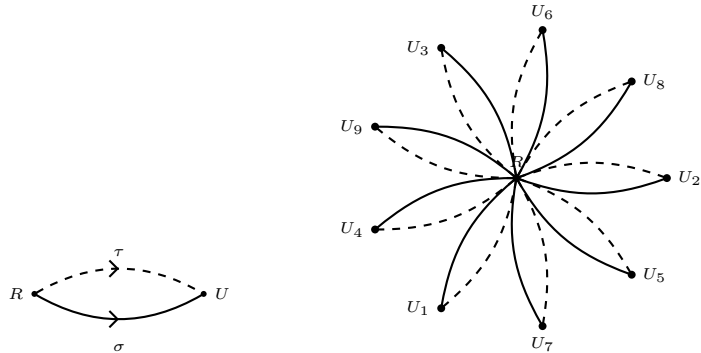
    \begin{center}
\begin{tikzpicture}[scale=0.75]
\node[inner sep=0pt, label = left:{\tiny{$R$}}] (A) at (-1.5,0) {};
\node[inner sep=0pt, label = right:{\tiny{$U$}}] (B) at (1.5,0) {}; 
\node[inner sep=0pt, label = below:{\tiny{$\tau$}}] (T) at (0,1) {}; 
\node[inner sep=0pt, label = below:{\tiny{$\sigma$}}] (T) at (0,-0.7) {};

\fill (1.5, 0) circle (1.5pt);
\fill (-1.5, 0) circle (1.5pt);

\draw[thick, mid arrow, dashed] (A) to[bend left] (B);
\draw[thick, mid arrow] (A) to[bend right] (B);
\end{tikzpicture}
\hspace{1cm}
\begin{tikzpicture}[scale=1]
\node[inner sep=0pt, label = above:{\tiny{$R$}}] (A) at (0,0) {};
\node[inner sep=0pt, label = right:{\tiny{$U_2$}}] (C) at (2,0) {}; 
\node[inner sep=0pt, label = right:{\tiny{$U_8$}}] (C') at (2*0.766,2*0.64278) {}; 
\node[inner sep=0pt, label = right:{\tiny{$U_5$}}] (C'') at (2*0.766,2*-0.64278) {}; 
\node[inner sep=0pt, label = above:{\tiny{$U_6$}}] (B) at (2*0.173648,2*0.9848) {}; 
\node[inner sep=0pt, label = left:{\tiny{$U_3$}}] (B') at (2*-0.5,2*0.866) {}; 
\node[inner sep=0pt, label = left:{\tiny{$U_9$}}] (B'') at (2*-0.93969,2*0.34202) {}; 
\node[inner sep=0pt, label = left:{\tiny{$U_4$}}] (D) at (2*-0.93969,2*-0.34202) {}; 
\node[inner sep=0pt, label = left:{\tiny{$U_1$}}] (D') at (2*-0.5,2*-0.866) {}; 
\node[inner sep=0pt, label = below:{\tiny{$U_7$}}] (D'') at (2*0.173648,2*-0.9848) {}; 

\fill (2,0) circle (1.5pt);
\fill (2*0.766,2*0.64278) circle (1.5pt);
\fill (2*0.766,2*-0.64278) circle (1.5pt);
\fill (0,0) circle (1.5pt);
\fill (2*0.173648,2*0.9848) circle (1.5pt);
\fill (2*-0.5,2*0.866) circle (1.5pt);
\fill (2*-0.93969,2*0.34202) circle (1.5pt);
\fill (2*-0.93969,2*-0.34202) circle (1.5pt);
\fill (2*-0.5,2*-0.866) circle (1.5pt);
\fill (2*0.173648,2*-0.9848) circle (1.5pt);

\draw[thick,dashed] (A) to[bend left=20] (C');
\draw[thick] (A) to[bend right=20] (C');
\draw[thick,dashed] (A) to[bend left=20] (D');
\draw[thick] (A) to[bend right=20] (D');
\draw[thick,dashed] (A) to[bend left=20] (D'');
\draw[thick] (A) to[bend right=20] (D'');
\draw[thick,dashed] (A) to[bend left=20] (D);
\draw[thick] (A) to[bend right=20] (D);
\draw[thick,dashed] (A) to[bend left=20] (B');
\draw[thick] (A) to[bend right=20] (B');
\draw[thick,dashed] (A) to[bend left=20] (C'');
\draw[thick] (A) to[bend right=20] (C'');
\draw[thick,dashed] (A) to[bend left=20] (B'');
\draw[thick] (A) to[bend right=20] (B'');
\draw[thick, dashed] (A) to[bend left=20] (C);
\draw[thick] (A) to[bend right=20] (C);
\draw[thick, dashed] (A) to[bend left=20] (B);
\draw[thick] (A) to[bend right=20] (B);
\end{tikzpicture}
\end{center}
In the next step we will get a flower with $27$ leaves. Thus, for $X_n$ we obtain 
\[
\ord_p(\vert \Jac(X_n)\vert)=\ord_p(2)p^{2n}.
\]
If we compute the characteristic ideal using the matrices 
\[D = \begin{bmatrix}
    2 & 0 \\
    0 & 0
\end{bmatrix} \textrm{ and } B = \begin{bmatrix}
    0 & 0 \\ 
    \tau^{-1}+\sigma^{-1} & -1
\end{bmatrix},
\]
we obtain that the characteristic ideal is given by
\[\det(D-B)=2. \]
If we let $R$ only ramify at the group generated by $\tau$, i.e. $\langle \tau\rangle$, we obtain 
\[
B=\begin{bmatrix}
    0&0\\\tau^{-1}+\sigma^{-1}&-(\tau-1)
\end{bmatrix}.
\]
In this case the characteristic ideal is given by $2T$.
{Here, we are identifying the variable $T$ with $\tau - 1$.}
\end{example}

\subsection{Growth Formula}

In this section, the aim is to prove an analogue of Iwasawa's growth number formula for the number of spanning trees at the $n$-th layer of the $\Z_p^d$-tower.
{For this, we need to appeal to the theory of `structures' discussed in \cite{cuoco-monsky}.
For the convenience of the reader we have included some details and repeated some proofs but at various places we refer the reader to the original text.}

\subsubsection{Abstract results on \texorpdfstring{$\Lambda$}{}-modules}

We first need to obtain some results on general $\Lambda$-modules.

\begin{definition}
\label{structure}
Let $M$ be a $\Lambda$-module.
A \emph{structure} on $M$ consists of an integer $z$ and tuples $(I_j,M_j)$ $1\le j\le z$, where the $M_j$ are submodules of $M$ and the $I_j$ are subgroups of $\mathcal{G}$. 
For all $n$, set
\[
A_n=(\omega_n(T_1),\dots, \omega_n(T_d))M+\sum_{j=1}^{z}\bomega_{j,n}M_j,
\]
where $\bomega_{j,n}$ is defined {as in \eqref{omega v n defn}} with respect to the subgroup $I_j$. 
With this notation, define $\mathcal{M}_n=M/A_n$.
\end{definition}
\begin{lemma}
\label{ex-const}
There is a constant $c$ such that $p^{2dn+c}\mathcal{M}_n[p^\infty]=0$.
\end{lemma}

\begin{proof}
Even though our definition of structures differs from the one given in \cite{cuoco-monsky}, the proof of Theorem~4.5 in \emph{loc. cit.} carries over almost verbatim.
We repeat it for the convenience of the reader. 

Let $N$ be the submodule generated by all the $M_{{j}}$.
Let $W=\Omega^d$, where $\Omega$ is the group of $p$-power roots of unity.
For each $\zeta\in W$ there is a natural map $M/N\to (M/N)\otimes \Z_p[\zeta]$ induced by the valuation $\Lambda\to \Z_p[\zeta]$.

\smallskip
{For any $\Lambda$-module $X$ we set the notation $X_{\zeta} = X \otimes \Z_p[\zeta]$.}

\smallskip

By \cite[Lemma~2.6]{cuoco-monsky} there exists a  constant $c$ such that $p^cx$ annihilates $(M/N)_\zeta[p^\infty]$ and $M_\zeta[p^\infty]$ for all choices of $\zeta$.
Let $n_0$ be the index such that the rank of $I_{{j}}$ in $\mathcal{G}/\mathcal{G}^{p^n}$ is stable for all $n\ge n_0$ {and all $j$}.

{First, consider an element} $\zeta\in W$ such that $\ord(\zeta)>p^{n_0}$.
The evaluation $\Lambda\to \Z_p[\zeta]$ sends $\bomega_{j,n}$ to zero in this case {for all $j$}.
Thus,
\[
(\mathcal{M}_n)_{\zeta}[p^\infty]=M_\zeta[p^\infty]
\]
is annihilated by $p^c$.

It remains to estimate the contribution of those $\zeta\in W$ with $\ord(\zeta)\le p^{n_0}$.
For all $n\ge n_0$ we define
\[
\bomega_n {= \bomega_n(T_1, \ldots, T_d)} =\prod_{i=1}^d\frac{(T_i+1)^{p^n}-1}{(T_i+1)^{p^{n_0}}-1}.
\]
Note that $\bomega_n$ describes the trace from $\mathcal{G}/\mathcal{G}^{p^n}$ to $\mathcal{G}/\mathcal{G}^{p^{n_0}}$.

Let $N'$ be the image of $N$ in $M_\zeta$.
(We emphasize that $N'$ is not $N_\zeta$ because taking tensor product is not left exact.)
Then the image of $\bomega_n N$ in $M_\zeta$ is $p^{(n-n_0)d}N'$.
It follows that $(M/\bomega_n N)_\zeta$ and $(M/N)_\zeta$ have the same $\Z_p$-rank.
Furthermore $(M/\bomega _n N)_\zeta[p^\infty]$ is annihilated by $p^{nd+c}$. 
Furthermore, we have a natural surjection
\[
(M/\bomega{_n} N)_\zeta \longrightarrow (M/A_n)_\zeta \longrightarrow ´(M/N)_\zeta,
\]
where all three terms have the same $\Z_p$-rank.
Therefore, $(M/A_n)_\zeta[p^\infty]$ is annihilated by $p^{nd+2c}$. 

Summing over all $\zeta$ of order at most $p^n$ we obtain that 
\[
\bigoplus _{\substack{\zeta \in W\\ \ord(\zeta)\le p^n}}(M/A_n)_\zeta[p^\infty]
\]
is annihilated by $p^{nd+2c}$.
It now follows from \cite[Lemma~2.7]{cuoco-monsky} that $\mathcal{M}_n[p^\infty]$ is annihilated by $p^{2dn+2c}$.
\end{proof}

\subsubsection{Applications}
We will apply the previous result to the following module arising from our graph theoretic setting.
Write
\begin{equation}
\label{defn M}
M= \Div_\Lambda(X_\infty)/{\Pr}_\Lambda^{\unr}(X_\infty) {\simeq \Pic_{\Lambda}(X_\infty)}
\end{equation}
{where we proved the isomorphism in Corollary~\ref{cor:ideals}.}
Set $z=s-r$ and write $(I_j,M_j)=(I_{v_{j+r}},\Lambda P_{v_{j+r}})$ for $1 \le j \le s-r$.
Note that as $1\le j\le s-r$, the vertices $v_{j+r}$ are the \emph{ramified} ones in $X_\infty/X$.
This defines a structure in the above sense; by Theorem~\ref{thm:projection} we have
\begin{align*}
\Pic(X_n){\otimes \Z_p} &= \Div_{\Lambda}(X_\infty)/{\Pr}_\Lambda^{\unr}(X_\infty) + (\omega_n(T_1),\dots, \omega_n(T_d))\Div_\Lambda(X_\infty) + {\Pr}_n^{\ram}(X_\infty) \\
& = M/A_n.
\end{align*}

\begin{theorem}
\label{asymp formula}
Fix $d\geq 2$.
Let $X_\infty$ denote the derived graph $X(\Z_p^d,\mathcal{I},\alpha)$ and $X_\infty^{{\unr}}$ denote the derived graph $X(\Z_p^d,\alpha)$. Assume that all intermediate $X_n$ are connected.
Let $f$ be a generator of the characteristic ideal of $\Jac_\Lambda(X_\infty)$.
Then
\[
\ord_p(\vert \Jac(X_n)\vert) = \mu(f) p^{nd}+\lambda(f) np^{(d-1)n}+O(p^{(d-1)n}).
\]
\end{theorem}

\textcolor{blue}{}

\begin{remark}
Recall that the number of spanning trees of a graph $X$ is equal to the size of the Jacobian of $X$.
The above statement can hence be rewritten in terms of $\ord_p(\kappa(X_n))$.
\end{remark}

\begin{proof}
{As before, let $I_v$ be the inertia subgroup associated with the vertex $v$.}
We arrange the ramified vertices in increasing order of the rank of the corresponding subgroups $I_v$.
In other words, we assume that $\rk(I_{v_j})\le \rk(I_{v_{j+1}})$ for $r+1\le j\le s$.
Let $s'\ge r+1$ be the index such that $\rk(I_{v_{{j}}})\ge 2$ for all $j\ge s'$.
For {ramified vertices $v_j$ with inertia rank 1 we denote the topological generator of $I_{v_j}$ by $\sigma_{v_j}$} and define\footnote{Note that $\widetilde{T}_j$ and $T_i$ are not related to each other per se, except that they lie in the same Iwasawa algebra.}
\[
\widetilde{T}_{{j}}=\sigma_{v_{{j}},1}-1  \text{ where } r+1\le {j} \le s'-1.
\]
Decompose $\Div_\Lambda(X_\infty)$ as follows
\[
\Div_\Lambda(X_\infty)=\mathfrak{M}_1\oplus \mathfrak{M}_2\oplus \mathfrak{M}_3,
\]
where 
\begin{itemize}
\item $\mathfrak{M}_1$ is $\Lambda$-free and the contribution to $\mathfrak{M}_1$ comes from the unramified vertices,
\item $\mathfrak{M}_2= \displaystyle\bigoplus_{j=1}^{s'-1}\Lambda/\widetilde{T}_j$ and the contribution to $\mathfrak{M}_2$ comes from those ramified vertices for which the rank of $I_v$ is 1. 
\item $\mathfrak{M}_3$ is pseudo-null and $\Z_p$-free.
\end{itemize}
For $1\le i\le 3$ we denote by $\overline{\mathfrak{M}_i}$ the image of $\mathfrak{M}_i$ in $\Pic_\Lambda(X_\infty)$.
Clearly, (by definition/construction) $\overline{\mathfrak{M}_3}$ is pseudo-null.
By \cite[Lemma~3.1]{cuoco-monsky}, 
\[
\rk_p(\pi_n(\overline{\mathfrak{M}_3}))\le \rk_p(\overline{\mathfrak{M}_3}/(\omega_n(T_1),\dots,\omega_n(T_d))\mathfrak{\overline{M_3}})=O(p^{n(d-2)}).
\]
By Lemma~\ref{ex-const}, we can thus conclude that
\[
\ord_p(\vert \pi_n(\overline{\mathfrak{M}_3})[p^\infty]\vert )=O(2nd p^{n(d-2)}).
\]

Set $M'=M/\overline{\mathfrak{M}_3}$ where $M$ is defined as in \eqref{defn M}.
Define $M'_j$ as the image of the submodule $M_j = {\Lambda P_{v_{j+r}}}$ in $M'$.
Set $z'=s'-r-1$; the tuples $(I_{v_{j+r}}, M'_j)$ for $1\le j\le z'$ define a structure on $M'$.
Let $\mathcal{M}'_n$ be the module constructed in Definition~\ref{structure} associated to this structure.
The structure defined on $M'$ is an admissible structure in the sense of \cite[Definitions~4.1 and 4.6]{cuoco-monsky}.
{Writing $f$ to denote a generator of the characteristic ideal of $M = \Pic_{\Lambda}(X_\infty)$, it follows from}
\cite[Theorems~4.12 and 4.13]{cuoco-monsky} and using the estimate for the contribution of $\overline{\mathfrak{M}_3}$ that
\[
\ord_p(\vert \mathcal{M}_n[p^\infty])\vert)=\ord_p(\vert \mathcal{M}'_n[p^\infty]\vert )+O(p^{n(d-1)})=\mu {(f)} p^{nd}+\lambda {(f)} np^{(d-1)n}+O(p^{(d-1)n}).
\]
{We have $\Char_{\Lambda}(\Pic_{\Lambda}(X_\infty)) = \Char_{\Lambda}(\Jac_{\Lambda}(X_\infty))$ since $d\geq 2$.}
The result follows since $M$ and $M'$ have the same characteristic ideal {(because $\overline{\mathfrak{M}_3}$ is pseudo-null)}.
\end{proof}

\begin{remark}
{The second named author and S. Kleine obtained a growth formula for unramified $\Z_p^d$-towers in \cite{Kleine-Mueller4}.
In contrast to the above result the result in \emph{loc. cit.} does not involve an error term.
The main difference is that the situation of branched coverings requires the use of structures in the above sense.
For unramified coverings this difficulty does not occur which allows the authors of \emph{loc. cit.} to derive an explicit formula without an error term. }
\end{remark}

\section{Results on Dual Graphs}
\label{sec: results on dual graphs}

\begin{definition}
Let $X$ be a planar graph.
{A} \emph{dual graph} of $X$, denoted by $X^\vee$, is a graph that has a vertex for each face of $X$ and an edge for each pair of faces in $X$ that are separated from each other by an edge, and a self-loop when the same face appears on both sides of an edge.
\end{definition}

Each edge $e$ of $X$ has a corresponding dual edge, with endpoints the dual vertices corresponding to the faces on either side of $e$.
{Recall that the notion of a dual graph is not unique and depends on the choice of embedding.
In what follows, whenever we start with a planar graph $X$, we will fix a choice of embedding even if we do not explicitly mention it.
Therefore, we will talk about `the' dual graph (with respect to this choice of embedding).
In the same vein, when we construct a (planar) covering $Y$ of $X$ we will again choose an embedding of $Y$ and talk about `the' dual of $Y$.}

\subsection{Jacobian of Dual Graphs}
Let $\cG$ be any uniform pro-$p$ $p$-adic Lie extension.
We first observe that the dual graphs of each layer of a $\cG$-tower of planar graphs need not be a $\cG$-tower.
{All $\mathcal{G}$-towers are still assumed to come from a voltage assignment.}

\begin{nexample}
Recall the tower of graphs in Example~\ref{ramified version of 3.4}.
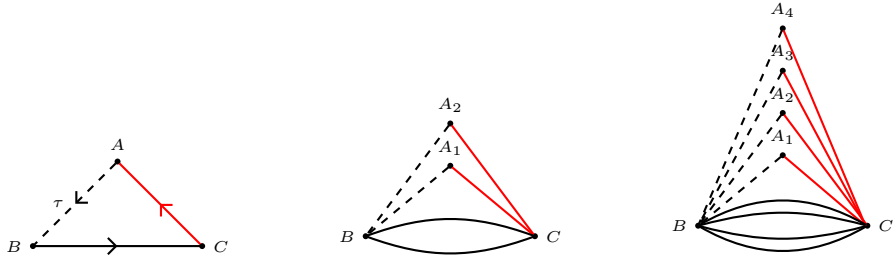
\begin{center}
\begin{tikzpicture}[scale=0.75]
\node[inner sep=0pt,  label = above:\tiny{$A$}] (A) at (0,1.5) {};
\node[inner sep=0pt,  label = left:\tiny{$B$}] (B) at (-1.5,0) {}; 
\node[inner sep=0pt,  label = right:\tiny{$C$}] (C) at (1.5,0) {}; 
\node[inner sep=0pt,  label = left:\tiny{$\tau$}] (T) at (-0.75,0.75) {};

\fill (0,1.5) circle (1.5pt);
\fill (1.5,0) circle (1.5pt);
\fill (-1.5,0) circle (1.5pt);

\draw[thick, mid arrow, dashed] (A) to (B);
\draw[thick, mid arrow, red] (C) to (A);
\draw[thick, mid arrow] (B) to (C);
\end{tikzpicture}
\hspace{1cm}
\begin{tikzpicture}[scale=0.75]
\node[inner sep=0pt, label = above:\tiny{$A_1$}] (A) at (0,1.25) {};
\node[inner sep=0pt, label = left:\tiny{$B$}] (B) at (-1.5,0) {}; 
\node[inner sep=0pt, label = right:\tiny{$C$}] (C) at (1.5,0) {}; 
\node[inner sep=0pt, label = above:\tiny{$A_2$}] (A') at (0,2) {};

\fill (0,1.25) circle (1.5pt);
\fill (0,2) circle (1.5pt);
\fill (1.5,0) circle (1.5pt);
\fill (-1.5,0) circle (1.5pt);

\draw[thick,dashed] (A') to (B);
\draw[thick, red] (C) to (A');
\draw[thick,dashed] (A) to (B);
\draw[thick, red] (C) to (A);
\draw[thick] (B) to[bend left=20] (C);
\draw[thick] (B) to[bend right=20] (C);
\end{tikzpicture}
\hspace{1cm}
\begin{tikzpicture}[scale=0.75]
\node[inner sep=0pt, label = above:\tiny{$A_1$}] (A) at (0,1.25) {};
\node[inner sep=0pt, label = left:\tiny{$B$}] (B) at (-1.5,0) {}; 
\node[inner sep=0pt, label = right:\tiny{$C$}] (C) at (1.5,0) {}; 
\node[inner sep=0pt, label = above:\tiny{$A_2$}] (A') at (0,2) {};
\node[inner sep=0pt, label = above:\tiny{$A_3$}] (A'') at (0,2.75) {};
\node[inner sep=0pt, label = above:\tiny{$A_4$}] (A''') at (0,3.5) {};

\fill (0,1.25) circle (1.5pt);
\fill (0,2) circle (1.5pt);
\fill (0,2.75) circle (1.5pt);
\fill (0,3.5) circle (1.5pt);
\fill (1.5,0) circle (1.5pt);
\fill (-1.5,0) circle (1.5pt);

\draw[thick,dashed] (A') to (B);
\draw[thick, red] (C) to (A');
\draw[thick,dashed] (A'') to (B);
\draw[thick, red] (C.west) to (A'');
\draw[thick,dashed] (A''') to (B);
\draw[thick, red] (C.west) to (A''');
\draw[thick,dashed] (A) to (B);
\draw[thick, red] (C.west) to (A);
\draw[thick] (B) to[bend left] (C);
\draw[thick] (B) to[bend right] (C);
\draw[thick] (B) to[bend left=15] (C);
\draw[thick] (B) to[bend right=15] (C);
\end{tikzpicture}
\end{center}

We now draw the respective dual graphs.
The vertex called $O$ is the one corresponding to the `outer face'.
The other vertices are labelled as $I_{?}$.

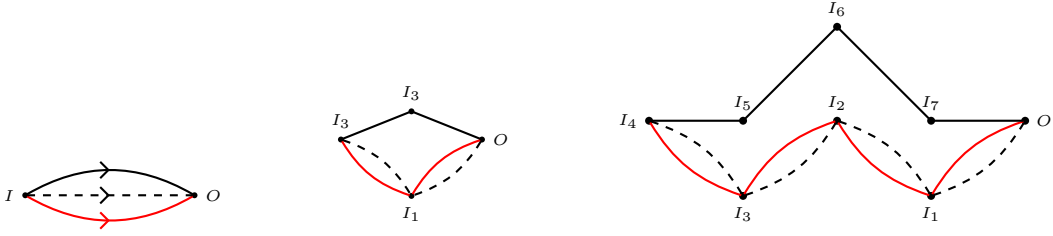
\begin{center}
\begin{tikzpicture}[scale=0.75]
\node[inner sep=0pt, label = left:\tiny{$I$}] (A) at (-1.5,0) {};
\node[inner sep=0pt, label = right:\tiny{$O$}] (B) at (1.5,0) {}; 

\fill (1.5, 0) circle (1.5pt);
\fill (-1.5, 0) circle (1.5pt);

\draw[thick, mid arrow] (A) to[bend left] (B);
\draw[thick, mid arrow, red] (A) to[bend right] (B);
\draw[thick, mid arrow, dashed] (A) to (B);
\end{tikzpicture}
\hspace{1cm}
\begin{tikzpicture}[scale=0.75]
\node[inner sep=0pt, label = below:\tiny{$I_1$}] (I1) at (0,0) {};
\node[inner sep=0pt, label = above:\tiny{$I_3$}] (I3) at (0,1.5) {}; 
\node[inner sep=0pt, label = right:\tiny{$O$}] (O) at (1.25,1) {};
\node[inner sep=0pt, label = above:\tiny{$I_3$}] (I2) at (-1.25,1) {}; 

\fill (0, 0) circle (1.5pt);
\fill (0, 1.5) circle (1.5pt);
\fill (1.25, 1) circle (1.5pt);
\fill (-1.25, 1) circle (1.5pt);

\draw[thick, dashed] (O) to[bend left=20] (I1);
\draw[thick, red] (O) to[bend right=20] (I1);
\draw[thick, dashed] (I2) to[bend left=20] (I1);
\draw[thick, red] (I2) to[bend right=20] (I1);
\draw[thick] (O) to (I3);
\draw[thick] (I3) to (I2);
\end{tikzpicture}
\hspace{1cm}
\begin{tikzpicture}[scale=1]
\node[inner sep=0pt, label = below:\tiny{$I_1$}] (I1) at (0,0) {};
\node[inner sep=0pt, label = below:\tiny{$I_3$}] (I3) at (-2.5,0) {};
\node[inner sep=0pt, label = left:\tiny{$I_4$}] (I4) at (-3.75,1) {};
\node[inner sep=0pt, label = above:\tiny{$I_7$}] (I7) at (0,1) {}; 
\node[inner sep=0pt, label = above:\tiny{$I_5$}] (I5) at (-2.5,1) {}; 
\node[inner sep=0pt, label = above:\tiny{$I_6$}] (I6) at (-1.25,2.25) {}; 
\node[inner sep=0pt, label = right:\tiny{$O$}] (O) at (1.25,1) {};
\node[inner sep=0pt, label = above:\tiny{$I_2$}] (I2) at (-1.25,1) {}; 

\fill (0, 0) circle (1.5pt);
\fill (-2.5, 0) circle (1.5pt);
\fill (-2.5, 1) circle (1.5pt);
\fill (-3.75, 1) circle (1.5pt);
\fill (0, 1) circle (1.5pt);
\fill (1.25, 1) circle (1.5pt);
\fill (-1.25, 1) circle (1.5pt);
\fill (-1.25, 2.25) circle (1.5pt);

\draw[thick, dashed] (I2) to[bend left=20] (I3);
\draw[thick, red] (I2) to[bend right=20] (I3);
\draw[thick, dashed] (O) to[bend left=20] (I1);
\draw[thick, red] (O) to[bend right=20] (I1);
\draw[thick, dashed] (I2) to[bend left=20] (I1);
\draw[thick, red] (I2) to[bend right=20] (I1);
\draw[thick, dashed] (I4) to[bend left=20] (I3);
\draw[thick, red] (I4) to[bend right=20] (I3);
\draw[thick] (O) to (I7);
\draw[thick] (I4) to (I5);
\draw[thick] (I6) to (I5);
\draw[thick] (I6) to (I7);
\end{tikzpicture}
\end{center}
It is clear from the picture that the second and third graphs are not (branched) coverings of the first one. 
\end{nexample}

Later on we will see examples where the dual graphs are actually coverings again.
The following theorem describes the Galois structure in this case.
{Given a tower $(X_n)_{n\in \mathbb{N}}$ of a planar graph $X$ we fix (once and for all) a planar representation of the graphs $X_n$.
As we have seen in the above example the corresponding dual graphs need not necessarily form a branched covering again.
In the following theorem we consider the case where the dual graphs actually build a branched covering.}

\begin{theorem}
\label{dual graph theorem}
Let $X_\infty/X$ be a $\mathcal{G}$-tower of planar connected graphs with intermediate layers denoted by $X_n$. 
{Fix a planar embedding for each $X_n$ and consider the dual graph $X^\vee_n$ at each layer.
Suppose} that $X_n^\vee/X^\vee$ is a branched covering.
Then $X_\infty^\vee/X^\vee$ is also a $\mathcal{G}$-tower of planar graphs with intermediate layers $X_n^\vee$ and the Galois action can be chosen to be compatible with the process of taking duals.
\end{theorem}

\begin{proof}
We fix a planar representation of $X$ and let $\pi_n\colon X_n\to X$ be the natural projection.
Given a directed edge $e$ of $X_n$ there are exactly two directed edges $e_1$ and $e_2$ in $X_n^\vee$ that intersect $e$.
We wish to choose one of them (as the dual edge) depending on the orientation of $e$.
We fix once and for all a choice $e^\vee$ for $X$ and define  $e_i=e^\vee$ if  $\pi_n(e)^\vee=\pi_n(e_i)$. 

Define the action of $\mathcal{G}_n$ on the set of directed edges $\E(X_n^\vee)$ by 
\[
g \cdot e^\vee=(g\cdot e)^\vee.
\]
We extend the action to the vertices of $X_n^\vee$ by defining $g(o(e^\vee))=o(ge^\vee)$ and $g t(e^\vee)=t(g e^\vee)$. Thus, $g$ extends naturally to a morphism of graphs.
This defines a free action without inversion, because the action of $\mathcal{G}_n$ on $\E(X_n)$ is a free action without inversion. \newline

\emph{Claim:}
This action is compatible with the tower.\newline

\emph{Justification:}
Let $e\in \E(X_n^\vee)$ and let $\pi_{n,n-1}\colon X_n^\vee\to X_{n-1}^\vee$ be the natural projection.
By abuse of notation we also denote the projection $X_n\to X_{n-1}$ by $\pi_{n,n-1}$.
Let $\psi_{n,n-1}\colon \cG_n\to \cG_{n-1}$ be the corresponding projection on the groups.
We need to show
\[
\pi_{n, n-1}(g\cdot e^\vee)=\psi_{n,n-1}(g) \cdot \pi_{n,n-1}(e^\vee).
\]
By construction,
\[
\pi_{n,n-1}(g \cdot e^\vee) = \pi_{n,n-1}((g\cdot e)^\vee) = (\pi_{n,n-1}(g\cdot e))^\vee = (\psi_{n,n-1}(g) \cdot \pi_{n,n-1}(e))^\vee=\psi_{n,n-1}(g) \cdot \pi_{n,n-1}(e^\vee).
\]
This completes the proof of the claim and the theorem.
\end{proof}

We elucidate this via an example.

\begin{example}
\label{example:dual}
Let $X$ be the following graph with $4$ vertices. 
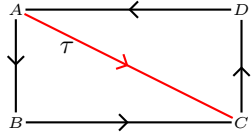
\begin{center}
\begin{tikzpicture}[scale=1]
\node[inner sep=0pt, shape = circle] (A) at (-1.5,1.5) {\tiny{$A$}};
\node[inner sep=0pt, shape = circle] (B) at (-1.5,0) {\tiny{$B$}}; 
\node[inner sep=0pt, shape = circle] (C) at (1.5,0) {\tiny{$C$}}; 
\node[inner sep=0pt, shape = circle] (D) at (1.5,1.5) {\tiny{$D$}}; 
\node[inner sep=0pt, shape = circle] (T) at (-0.82,1) {$\tau$};

\draw[thick, mid arrow] (A.south) to (B.north);
\draw[thick, mid arrow] (B.east) to (C.west);
\draw[thick, mid arrow] (C.north) to (D.south);
\draw[thick, mid arrow] (D.west) to (A.east);
\draw[thick, mid arrow, red] (A) to (C);
\end{tikzpicture}
\end{center}
Let $\tau$ be a topological generator of $\Z_p$. Let $X_\infty$ be the derived graph of this voltage assignment and let $X_n$ be the intermediate coverings. By Theorem \ref{thm:planar} all $X_n$ are planar. The graph $X_n$ has the following form (picture for $p=2$ and $n=1$).
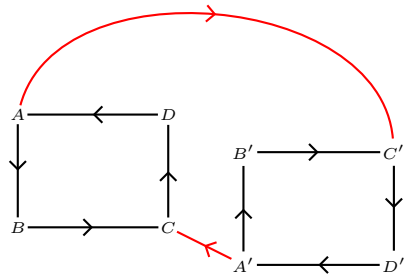
\begin{center}
\begin{tikzpicture}[scale=1]
\node[inner sep=0pt, shape = circle] (A) at (-1.5,1.5) {\tiny{$A$}};
\node[inner sep=0pt, shape = circle] (B) at (-1.5,0) {\tiny{$B$}}; 
\node[inner sep=0pt, shape = circle] (C) at (.5,0) {\tiny{$C$}}; 
\node[inner sep=0pt, shape = circle] (D) at (.5,1.5) {\tiny{$D$}}; 

\node[inner sep=0pt, shape = circle] (B') at (1.5,1) {\tiny{$B'$}};
\node[inner sep=0pt, shape = circle] (A') at (1.5,-0.5) {\tiny{$A'$}}; 
\node[inner sep=0pt, shape = circle] (D') at (3.5,-0.5) {\tiny{$D'$}}; 
\node[inner sep=0pt, shape = circle] (C') at (3.5,1) {\tiny{$C'$}}; 

\draw[thick, mid arrow] (A.south) to (B.north);
\draw[thick, mid arrow] (B.east) to (C.west);
\draw[thick, mid arrow] (C.north) to (D.south);
\draw[thick, mid arrow] (D.west) to (A.east);
\draw[thick, mid arrow, red] (A) to[bend left=80] (C');
\draw[thick, mid arrow, red] (A') to (C);
\draw[thick, mid arrow] (A') to (B');
\draw[thick, mid arrow] (B') to (C');
\draw[thick, mid arrow] (C') to (D');
\draw[thick, mid arrow] (D') to (A');
\end{tikzpicture}
\end{center}
The dual graphs of the base graph has the following shape:
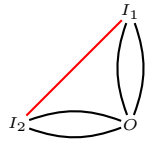
\begin{center}
    \begin{tikzpicture}
        \node[inner sep=0pt, shape = circle] (A) at (0,1.5) {\tiny{$I_1$}};
\node[inner sep=0pt, shape = circle] (B) at (-1.5,0) {\tiny{$I_2$}}; 
\node[inner sep=0pt, shape = circle] (C) at (0,0) {\tiny{$O$}}; 

\draw[thick,red] (A) to (B);
\draw[thick] (A) to [bend left=20] (C);
\draw[thick] (A) to [bend right=20] (C);
\draw[thick] (B) to [bend left=20] (C);
\draw[thick] (B) to [bend right=20] (C);
    \end{tikzpicture}
\end{center}
If we look at the dual of the first layer for $p=2$ we obtain
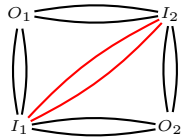
\begin{center}
    \begin{tikzpicture}
        \node[inner sep=0pt, shape = circle] (A) at (-1.5,1.5) {\tiny{$O_1$}};
\node[inner sep=0pt, shape = circle] (B) at (-1.5,0) {\tiny{$I_1$}}; 
\node[inner sep=0pt, shape = circle] (C) at (.5,0) {\tiny{$O_2$}}; 
\node[inner sep=0pt, shape = circle] (D) at (.5,1.5) {\tiny{$I_2$}}; 

\draw[thick, red] (B) to[bend right=10] (D);
\draw[thick, red] (B) to[bend left=10] (D);
\draw[thick] (B) to[bend right=10] (C);
\draw[thick] (B) to[bend left=10] (C);
\draw[thick] (A) to[bend right=10] (B);
\draw[thick] (A) to[bend left=10] (B);
\draw[thick] (C) to[bend right=10] (D);
\draw[thick] (C) to[bend left=10] (D);
\draw[thick] (A) to[bend right=10] (D);
\draw[thick] (A) to[bend left=10] (D);
    \end{tikzpicture}
\end{center}
The second graph is a branched covering of the first one, where the voltage assignment is trivial on all black edges and $\tau$ on the red one.
For higher $n$, the graph $X_n$ will consist of $p^n$ squares connected by red edges as drawn in the example $p=2$ and $n=1$ as above.
The dual graph has one vertex for each square denoted by $O_1, \dots ,O_{p^n}$ and two additional vertices created by the red edges of $X_n$.
We will denote these vertices by $I_1$ and $I_2$.
Each $O_i$ is connected to each $I_j$ by two black edges. 
Furthermore there are $p^n$ red edges connecting $I_1$ to $I_2$.
Thus, $X_n^\vee$ is the derived graph of the voltage assignment that is totally ramified in $I_1$ and $I_2$ and assigns $\tau$ to the red edge in $X_0^\vee$ and $1$ to all other edges.
\end{example}

{The following result gives a necessary condition for the dual graphs to indeed form a branched covering.}

\begin{corollary}{
Let $(X_n)_{n\in \mathbb{N}}$ be a $\Z_p$-tower of planar graphs and fixed planar embeddings of the $X_n$.
Assume that $(X^\vee_n)$ also builds a $\Z_p$-tower.
Let $\mathsf{v}_r$ and $\overline{\mathsf{v}}_r$ be the number of ramified vertices in the respective towers for {$n\gg0$}.
Then $0\le \mathsf{v}_r\le 2$ and $\overline{\mathsf{v}}_r = 2 - \mathsf{v}_r$.
}
\end{corollary}
\begin{proof}
{Consider the $n$-th layer of the $\Z_p$-tower of planar graphs.
Let $\mathsf{v}_n$ be the number of vertices, $\mathsf{e}_n$ be the number of edges, and $\mathsf{f}_n$ be the number of faces. As $X_n$ is planar,
\[
\mathsf{f}_n = 2 + \mathsf{e}_n - \mathsf{v}_n.
\]
Let $\mathsf{v}_r$ be the number of ramified vertices and $\mathsf{v}_u$ be the number of unramified vertices in $X_{n_0}$ where $X_{n_0}$ is the index occurring in Lemma~\ref{lemma:number of vertices}.
Then for all $n\ge n_0$, it follows from an application of Lemma~\ref{lemma:number of vertices} that
\[
\mathsf{f}_n = 2 - \mathsf{v}_r + p^{n-n_0}(\mathsf{e}_0 - \mathsf{v}_u).
\]
Without loss of generality, assume that $n_0$ is also the index determined by Lemma~\ref{lemma:number of vertices} for the dual tower, otherwise we just enlarge $n_0$.
Let $\overline{\mathsf{v}}_{r}$ be the number of ramified vertices in the tower $(X_n^\vee)_{n\ge n_0}$ and $\overline{\mathsf{v}}_{u}$ be the number of unramified vertices.
As $\mathsf{f}_n$ is the number of vertices of $X_n^\vee$ we obtain that for $n \ge n_0$,
\[
2 - \mathsf{v}_r + p^{n-n_0}(\mathsf{e}_0 - \mathsf{v}_u) = p^{n-n_0}\overline{\mathsf{v}}_u + \overline{\mathsf{v}}_r.
\]
This implies that $\overline{\mathsf{v}}_u = \mathsf{e}_0 - \mathsf{v}_u$ and $2-\mathsf{v}_r=\overline{\mathsf{v}}_r$.}
\end{proof}

\begin{definition}
Let $X=(V,\E)$ be an {un}directed graph.
A \emph{dart assignment} is a function
\[
{\phi}\colon \E \longrightarrow \Z,
\]
such that ${\phi}(e)+{\phi}(\overline{e})=0$, where $\overline{e}$ denotes the inverse of $e$.
Set $F(X)$ to denote the faces of $X$, and $f^\circ(X)$ to denote the boundary of $F(X)$ in a counterclockwise orientation.
{For any dart-assignment $\phi$, define the associated \emph{vertex-configuration} and \emph{face-configuration} as follows:}
\begin{align*}
\partial {\phi} & =\sum_{v\in V(X)}\sum_{e\in \E^o_v(X)}{\phi}(e)v \\
\partial^* {\phi} & =\sum_{f\in F(X)} \sum_{e\in f^\circ(X)} {\phi}(e)f.
\end{align*}
\end{definition}

\begin{theorem}
\label{Jac are isom}
Keep the notation introduced above and assume that $(X_n^\vee)_{n\in \mathbb{N}}$ is a tower, as well.
Then as Galois modules 
\[
\Jac(X_n) \simeq \Jac(X_n^\vee).
\]
\end{theorem}
    
\begin{proof}
As abstract groups, the isomorphism of the Jacobian groups is proven in \cite[Theorem~2]{cori2000sandpile}.
Let $u$ be an element in ${\Div}^0(X)$.
By \cite[Proposition~3.2]{cori2000sandpile} there exists a dart assignment ${\phi}$ such that $\partial {\phi} =u$.
It follows that $\partial^*({\phi})\in {\Div}^0(X^\vee)$.
By \cite[proof of Theorem~2]{cori2000sandpile} this assignment induces a well-defined bijective homomorphism
\[
\theta\colon {\Jac}(X)\longrightarrow {\Jac}(X^\vee).
\]
It remains to check that $\theta$ is a $G$-homomorphism.
Let $x\in \Jac(X)$ and let $u$ be a lift in ${\Div}^0(X)$.
Then
\begin{align*}
g u=g\sum_{v\in V(X)} u_vv=\sum_{v\in V(X)}u_vg(v)=\sum_{v\in V(X)}u_{g^{-1}v}v & =\sum_{v\in V(X)}\sum_{e\in \E^o_{g^{-1}v}}{\phi}(e)v\\
&=\sum_{v\in V(X)}\sum_{e\in \E^o_v}{\phi}(g^{-1}e)v\\
& =\sum_{v\in V(X)}\sum_{e\in \E^o_v}(g{\phi})(e)v\\
& = \partial(g{\phi}).
\end{align*}
Therefore,
\[
\theta(gx)=\partial^*(g{\phi})+\Pr(X^\vee)\in {\Jac}(X^\vee).
\]
Furthermore,
\[
\partial^*(g{\phi})=\sum_{f\in F(X)}\sum_{e\in f^\circ(X)}{\phi}(g^{-1}e)f=\sum_{f\in F(X)}\sum_{e\in f^\circ(X)}{\phi}(e)gf=g\partial^*({\phi}).
\]
Therefore
\[
g\theta(x)=\theta(gx). \qedhere
\] 
\end{proof}

\begin{remark}
It is crucial that we work with the Jacobian group here and not with the whole Picard group.
For Picard groups the above statement is in general not true.
Indeed, consider a  graph $X$ with $1$ vertex and one loop. Let $X_\infty/X$ be an unramified $\Z_p$-tower of $X$ with intermediate layers $X_n$.
In this case $\Pic(X_n)\cong \Lambda/(T^2,\omega_n)$ where $\omega_n {= \omega_n(T)}$ was defined previously in Theorem~\ref{thm:projection}.
{We define $\nu_{n,0}=\omega_n/T$.}
For $n>0$, the dual graph $X_n^\vee$ has $2$ vertices and $p^n$ edges between these vertices.
In this case 
\[
\Pic(X_n^\vee)\cong \left((\Lambda/T)e_1\oplus(\Lambda/T)e_2\right)/(\nu_{n,0}(e_1-e_2)).
\]
These two modules are certainly not isomorphic as $\Lambda$-modules. But
\[\Jac(X_n^\vee)=\left((\Lambda/T)(e_1-e_2)\right)/\nu_{n,0}(e_1-e_2)\cong\left((\Lambda/T)(e_1-e_2)\right)/p^n(e_1-e_2)\cong T\Lambda/(T^2,\omega_n)\cong \Jac(X_n) \]
\end{remark}

Recall that if $X$ and $X^\vee$ are dual planar graphs and $\mathcal{T}$ is a spanning tree for $X$, then the complement of the edges dual to $\mathcal{T}$ is a spanning tree for $X^\vee$, see \cite{biggs1971spanning}.
It now follows that
\[
\ord_p(\kappa(X_n)) = \ord_p(\kappa(X_n^\vee)).
\]
In other words, the generalized Iwasawa invariants associated with the finite layers above the graph $X$ and $X^\vee$ must coincide.
Another way to see the above equality is to recall the fact that $\kappa(X_n) = \# \Jac(X_n)$, the equality is immediate from Theorem~\ref{Jac are isom}.

By Theorem~\ref{Jac are isom}, we can conclude that $\Jac_{\Lambda}(X_\infty) \simeq \Jac_{\Lambda}(X_\infty^\vee)$ as $\Lambda$-modules.
Therefore,
\[
{\Char}_{\Lambda}(\Jac_{\Lambda}(X_\infty)) = {\Char}_{\Lambda}(\Jac_{\Lambda}(X_\infty^\vee)).
\]

\begin{example}
Revisiting Example \ref{example:dual} we can now compute the characteristic ideal of $\Jac(X_\infty)$ by computing the characteristic ideal of $\Jac(X_\infty^\vee)$.
While the computation of $\Jac(X_\infty)$ involves the computation of a $5\times 5$ matrix, the computation of the characteristic ideal of $\Jac(X_\infty^\vee)$ only involves the computation of the determinant of the following $3\times 3$ matrix
\[
\frac{1}{T}\det\left(\begin{bmatrix}
        4&0&0\\
        -2&T&0\\
        -2&0&T
    \end{bmatrix}\right)=4T.
\]
\end{example}

\subsection{Applications}

\label{sec app}

\begin{corollary}
\label{cor:app}
Assume that $(X_n)_{n\in \mathbb{N}}$ is an \emph{unramified} $\Z_p$-tower of planar graphs.
Assume that for each $n$, $X_n^\vee/X^\vee$ is a branched covering. 
Then $T\mid \Char_\Lambda({\Jac}_\Lambda(X_\infty))$.
\end{corollary}

\begin{proof}
Let $\mathsf{e}_n$ be the number of undirected edges, $\mathsf{v}_n$ be the number of vertices, and $\mathsf{f}_n$ be the number of faces of $X_n$.
As all graphs are planar, the Euler characteristic formula asserts that for all $n$
\begin{align*}
2 & = \mathsf{v}_n - \mathsf{e}_n + \mathsf{f}_n\\
& = p^n(\mathsf{v}_0-\mathsf{e}_0)+\mathsf{f}_n \textrm{ since } X_n/X_0 \textrm{ is unramified.}
\end{align*}
Rearranging the terms,
\[
\mathsf{f}_n = 2 - p^n(\mathsf{v}_0 - \mathsf{e}_0) = 2 + p^n(\mathsf{f}_0-2).
\]
    
By Lemma \ref{lemma:number of vertices}, we know that $X_n^\vee$ has exactly two totally ramified vertices and that all other vertices are unramified.
By construction, the matrix $D-B$ has the form
\[\begin{bmatrix}
    *&*&\dots &*&*&*\\
    \dots&\dots&\dots &\dots &*&*\\
    \dots&\dots&\dots &\dots &*&*\\
    \dots&\dots&\dots &\dots &*&*\\
    \dots&\dots&\dots &\dots &*&*\\
    0&0&0&0&T&0\\
    0&0&0&0&0&T
\end{bmatrix}\]
The determinant of this matrix is clearly divisible by $T^2$. 
The claim follows from \cite[Theorem~5.9]{GV24} and the fact that ${\Char}_{\Lambda}(\Jac_{\Lambda}(X_\infty)) = {\Char}_{\Lambda}(\Jac_{\Lambda}(X_\infty^\vee))$. 
\end{proof}

\begin{example}
The above corollary can be applied to the unramified cover given in Example~\ref{example:dual}.
Indeed, if one computes the characteristic ideal of $\Jac(X_\infty)$ one obtains $4T$ which is clearly divisible by $T$. 
\end{example}

\begin{corollary}
{Let $X$ be a finite connected graph and let $\alpha$ be a single voltage assignment.
Let $(X_n)_{n\in \mathbb{N}}$ be the unramified $\Z_p$-tower obtained from $\alpha$.
Then $T\mid \Char_\Lambda(\Jac_\Lambda(X_\infty))$.}
\end{corollary}

\begin{proof}
{Let $e_0$ be the edge of $X$ having a non-trivial voltage assignment.
We fix an embedding of $X$ such that $e_0$ lies on the outer face.}

{By Theorem~\ref{thm:planar} all graphs $X_n$ are planar.
For each $X_n$ we fix an embedding in the plane.
Let $X_n^\vee $ be the corresponding dual graphs.
If we can show that $X_n^\vee/X^\vee $ is a branched covering, the claim will follow from Corollary~\ref{cor:app}. }

{Let $X'$ be the graph obtained from $X$ by deleting the edge $e_0$.
Then $X_n$ consists of $p^n$ copies of $X'$ (the different sheets) and the additional edges that restrict to $e_0$.
We draw $X_n$ as $p^n$-copies of $X'$ next to each other and then add the additional edges.
The graph $X_n$ has three different types of faces: 
\begin{enumerate}
    \item faces that lie completely in one sheet of $X_n$.
    \item one outer face
    \item one face that is neither of the first two.
\end{enumerate}
Let
\[
\pi_n\colon X_n \longrightarrow X
\]
be the covering of graphs.
This induces a surjective map of graphs
\[
\pi^\vee_n\colon X_n^\vee \longrightarrow X^\vee.
\]}
{Consider a face $F_{{n}}$ of $X_n$ of type (1), and fix a path $\mathsf{P}_{F_n}$ around its boundary counterclockwise.
Note that then $\pi_n(\mathsf{P}_{F_{{n}}})=\mathsf{P}_{F'}$ where $F'$ is a face of $X'$ (and can also be viewed as a face of $X$).
Every edge of $F_n$ has a corresponding edge in $X_n^\vee$ by construction of the dual graph.
Note that each edge starting or ending in $F_{{n}}$ corresponds to (projects to) exactly one edge in $X^\vee$.
In particular, the edges of a fixed face $F_{n}$ of type (1) and edges of $F'$ are in one to one correspondence.
We further note that the face $F'$ is a vertex in $X^\vee$.}

{Now, consider a face $F_{n}$ of $X_n$ type (2).
Observe that $\pi_n(\mathsf{P}_{F_{n}})=p^n \mathsf{P}_O$, where $O$ is the outer face of $X$.
If $F_n$ is a face of type (3), then $\pi_n(\mathsf{P}_{F_n})=p^n \mathsf{P}_{F_0}$, where $F_0$ is the face of $X$ having $e_0$ on its boundary but is not the outer face. 
Thus there is a $p^n$-to-1 correspondence between edges starting/ending at $F_{n}$ and edges starting and ending at $O$.
Thus, $\pi^\vee_n$ is indeed a branched covering.}
\end{proof}

\bibliographystyle{amsalpha}
\bibliography{references}

\end{document}